\documentclass[final]{siamltex}
\usepackage{graphicx,bm,hyperref,amssymb,amsmath}   % siamltex hates amsthm

\newcommand{\bi}{\begin{itemize}}
\newcommand{\ei}{\end{itemize}}
\newcommand{\ben}{\begin{enumerate}}
\newcommand{\een}{\end{enumerate}}
\newcommand{\be}{\begin{equation}}
\newcommand{\ee}{\end{equation}}
\newcommand{\bea}{\begin{eqnarray}} 
\newcommand{\eea}{\end{eqnarray}}
\newcommand{\ba}{\begin{align}} 
\newcommand{\ea}{\end{align}}
\newcommand{\bse}{\begin{subequations}} 
\newcommand{\ese}{\end{subequations}}
\newcommand{\bc}{\begin{center}}
\newcommand{\ec}{\end{center}}
\newcommand{\bfi}{\begin{figure}}
\newcommand{\efi}{\end{figure}}
\newcommand{\ca}[2]{\caption{#1 \label{#2}}}
\newcommand{\ig}[2]{\includegraphics[#1]{#2}}
\newcommand{\bmp}[1]{\begin{minipage}{#1}}
\newcommand{\emp}{\end{minipage}}
\newcommand{\bp}{\begin{proof}}
\newcommand{\ep}{\end{proof}}
\newcommand{\ie}{{\it i.e.\ }}

\newcommand{\tbox}[1]{{\mbox{\tiny #1}}}
\newcommand{\mbf}[1]{{\mathbf #1}}
\newcommand{\half}{\mbox{\small $\frac{1}{2}$}}

\newcommand{\C}{\mathbb{C}}
\newcommand{\N}{\mathbb{N}}
\newcommand{\R}{\mathbb{R}}
\newcommand{\Z}{\mathbb{Z}}

\newcommand{\vt}[2]{\left[\begin{array}{r}#1\\#2\end{array}\right]} % 2-col-vec
\newcommand{\bigO}{{\mathcal O}}

\DeclareMathOperator{\cond}{cond}
\DeclareMathOperator{\sinc}{sinc}
\newcommand{\pbar}{\overline{p}}
\newcommand{\qbar}{\overline{q}}
\newcommand{\si}{\sigma}
\newcommand{\al}{\alpha}
\newcommand{\bt}{\beta}

\newtheorem{thm}{Theorem}

\newtheorem{lem}[thm]{Lemma}
\newtheorem{cor}[thm]{Corollary}
\newtheorem{pro}[thm]{Proposition}
\newtheorem{rmk}[thm]{Remark}

\begin{document} % ===========================================

% Contiguous submatrices of the Fourier matrix are exponentially ill-conditioned
% (can't use since Moitra).
% On the exponential rate of ill-conditioning of contiguous Fourier submatrices
%
% On the exponential rate of ill-conditioning of contiguous submatrices of the Fourier matrix

\title{How exponentially ill-conditioned are contiguous submatrices of the Fourier matrix?}
% homage to Pan'16.

  \author{Alex H. Barnett}
\maketitle

\begin{abstract}
  We show that the condition number of any cyclically
  contiguous $p\times q$ submatrix of the
  $N\times N$ discrete Fourier transform (DFT) matrix
  is at least
  $$
  \exp \left( \frac{\pi}{2} \left[\min(p,q)- \frac{pq}{N}\right]\right)~,
  $$
  up to algebraic prefactors.
  That is, fixing any shape parameters $(\al,\bt):=(p/N,q/N)\in(0,1)^2$,
  the growth is $e^{\rho N}$ as $N\to\infty$ with rate
  $\rho = \frac{\pi}{2}[\min(\alpha,\beta)- \alpha\beta]$.
  Such Vandermonde system matrices arise in many applications, such as
  Fourier continuation, super-resolution, and diffraction imaging.
  Our proof uses the Kaiser--Bessel transform pair (of which we give a self-contained proof) and estimates on sums
  over distorted sinc functions,
  to construct a localized trial vector
  whose DFT is also localized.
  We warm up with an elementary proof of the above but with {\em half} the rate,
  via a periodized Gaussian trial vector.
  Using low-rank approximation of the kernel $e^{ixt}$,
  we also prove another lower bound $(4/e\pi \al)^q$, up to algebraic
  prefactors,
  %for $0<\bt\le\al < 4/e\pi$,
  which is stronger than the above for small $\al, \bt$.
  When combined, the bounds are within a factor of two of the
  numerically-measured empirical asymptotic rate, uniformly over $(0,1)^2$,
  and they become sharp in certain regions.
  However, the results are not asymptotic: they apply to
  essentially all $N$, $p$, and $q$, and with all constants explicit.
\end{abstract}

% IIIIIIIIIIIIIIIIIIIIIIIIIIIIIIIIIIIIIIIIIIIIIIIIIIIIIIIIIIIIIIIIIIIIIIIIIIII
\section{Introduction and main results}

The size-$N$ discrete Fourier transform (DFT) matrix $F$ has elements
\be
F_{jk} = e^{2\pi i jk/N},  \qquad j,k\in\Z, \quad
-N/2 \le j,k < N/2~,
\label{F}
\ee
where the row and column index sets should be taken as $N$-periodic;
we center them on zero for convenience later.%
\footnote{Note that when $N$ is even each set is $\{-N/2,\dots,N/2-1\}$,
  whereas when $N$ is odd it is $\{-(N-1)/2,\dots,(N-1)/2\}$.}
%Here the index set may be any cyclic permutation of $\{0,1,\dots,N-1\}$.
Taking the DFT of a vector in $\C^N$,
for instance via the fast Fourier transform (FFT) algorithm,
is equivalent to multiplication by $F$.
That $F$ is full rank and its matrix condition number, $\cond(F)$,
is $1$ follows from the unitarity of $F/\sqrt{N}$.
However, it is well known that contiguous submatrices
of $F$ of size $p\times q$ are
approximately low rank, with $\epsilon$-rank
of order $pq/N$ \cite{edelman99,oneil07,zhu17}.
% *** Slepian ? Rokhlin cites but I don't see it.
On the other hand, their condition number must be finite,
because any submatrix (or its adjoint) may be obtained by deleting columns
%(see, e.g., \cite[Thm~.1]{thompson72} interlacing; no L.I. pf easier).
of a Vandermonde, hence nonsingular, matrix
($z_j^k$ for the nodes $z_j = e^{2\pi i j/N}$).
Yet, Vandermonde matrices are suspected to be exponentially ill-conditioned
unless nodes are equispaced over the entire unit circle \cite{pan16}.
This leads one naturally to ask: how do
the condition numbers of Fourier submatrices behave?
Their growth has been established to be exponential
\cite[Thm.~3.1]{moitra}, and preliminary
study exists for the case $p=q=N/2$ \cite[Thm.~6.2 and Table~4]{pan16}.
Yet, what is the exponential {\em rate}, and how does it depend on $p$ and $q$?
Figure~\ref{f:nearsymm} illustrates this growth,
%for some small $N$ values,
and hints at a {\em universal} %growth
rate
%function
depending only on the scaled submatrix shape.
%shape parameters $p/N$ and $q/N$.
There appear to have been few rigorous answers to these
quite simple and fundamental questions.
%for which, however, there have been very few rigorous answers.

\bfi[t]  % fffffffffffffffffffffffffffffffffffffffffffffffffffffffffffffffff
%\ig{width=1.3in}{nearsymm1.eps}\ig{width=1.3in}{nearsymm2.eps}\ig{width=1.3in}{nearsymm3.eps}
\centering \ig{width=5.3in}{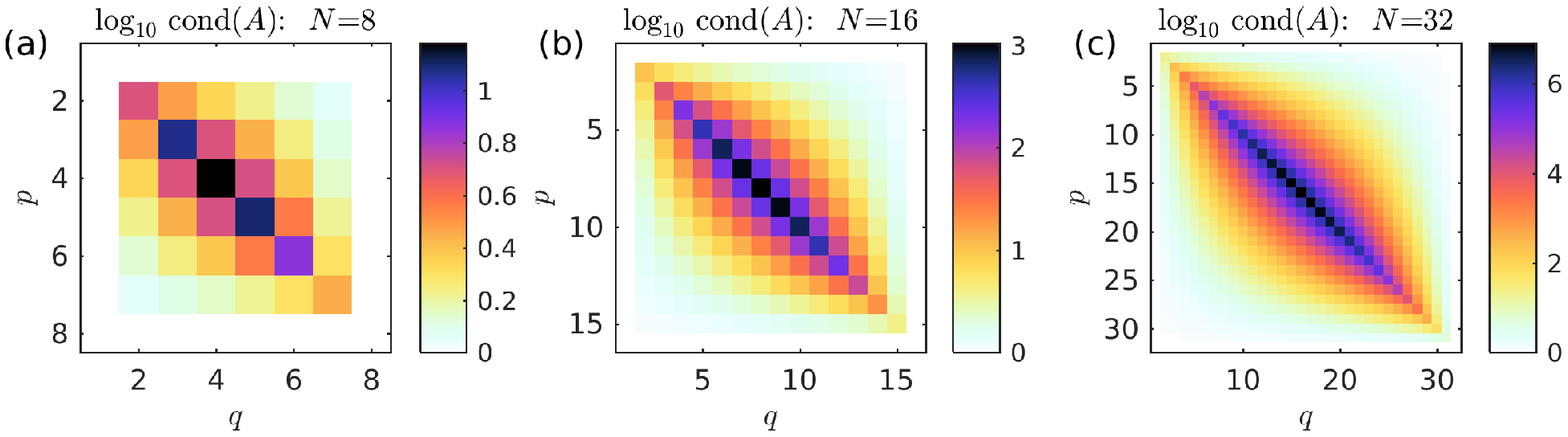}
%\vspace{-1ex}
\ca{Condition numbers of all possible submatrices $A$ of the size-$N$ Fourier
  matrix $F$, as a function of their size $p\times q$.
  Panels (a--c) compare this for three increasing $N$ values, which are
  still rather small.
  Each vertical ($p$) axis is oriented downwards to match the usual sense
  for matrices.
  The color scale is logarithmic, illustrating exponential growth both
  in $N$, and as $(p,q)$ tends to the diagonal.
  See Figure~\ref{f:rates} for a study of the $N\to\infty$ exponential rates
  and comparison to lower bounds proven in this work.
  The diagonal symmetry $p\leftrightarrow q$ is exact for each $N$, but
  the inversion symmetry $(p,q)\leftrightarrow(N-p,N-q)$ is only
  true asymptotically as $N\to\infty$
  (see Section~\ref{s:symm}).}{f:nearsymm}
\efi

% Apps
The conditioning of Fourier submatrices has consequences
in applications because it controls the numerical stability,
or noise amplification, of various function and image reconstruction problems:
%There are at least three areas with a direct application.
\ben
\item
  In Fourier extension (or continuation) methods
  \cite{boydfext,huybrechs10,matthysen16}, with applications including
  the numerical solution of PDEs \cite{albinFC11},
  Fourier series coefficients are solved by collocation
  on a grid covering a {\em fraction} of the periodic interval.
  In the common case of a uniform grid,
  the system matrix is a contiguous submatrix of $F$
  with shape parameter $\al:=p/N$ equal to the covered fraction,
  and $\bt:=q/N$ equal to this fraction divided by the oversampling.%
  \footnote{Specifically, $\al=1/T$ and
    $\bt = 1/\gamma T$ in the notation of \cite[\S 5]{adcock14}.}
  The apparent exponential ill-conditioning is well documented \cite{boydfext,brubeck19}
  \cite[Fig.~13]{adcock14}.
  However Adcock et al.~\cite{adcock14} %Adcock, Huybrechs and Mart\'in-Vaquero 
  concluded, ``there is no existing analysis [of submatrices of $F$]
  akin to that of Slepian's for the prolate matrix\dots''.
  Zhu et al.~\cite{zhu17} have since supplied one such missing piece; this
  paper supplies another.
  %This paper at least supplies one such missing piece.
\item
  Super-resolution imaging has importance
  in microscopy, astronomy, radar, and medicine
  (see \cite{donoho92,candes14,moitra,liliao18,batenkov20}
  and references within).
  Even when source locations are known, and
  other fascinating issues such as sparsity and clustering are put aside,
  a linear system must be solved for the amplitudes,
  given known Fourier series coefficients
 (in the discrete model)
  up to some bandlimit $p$.
  % \cite{moitra,liliao18,batenkov20}.
  % in the discrete model, not continuous of Donoho92
  A pathological arrangement for $q$
  such sources is a regular grid (``clump'')
  %(``consecutive atoms''
  %\cite{demanet15} is continuous case, like Donoho.
  with spacing some factor (the so-called SRF
  \cite{donoho92,candes14,batenkov20}) times
  smaller than the Nyquist spacing.
  The system matrix becomes a $p\times q$ Fourier submatrix, with
  $\al=1/\mbox{SRF}$,
  and the exponential blow-up
  of its conditioning is a fundamental obstacle in the presence of noise.
  %\cite{moitra,liliao18,batenkov20}.
  % takes form of 3rd row of our Table 1.  Or see prior results below.
%  such a ``clump'' is used to derive lower bounds on the 
%  and lower bounds on its exponential ill-conditioning are of interest
  %\cite[Thm.~1.4]{donoho92} is continuous Fourier data case
  % factor = SRF of Demanet.
  % Batenkov20: discrete case, ``rect'' Vandermonde.
\item
  In coherent X-ray diffraction imaging,
  the data are squared magnitudes of the
  % oversampled
  Fourier transform of
  an unknown image \cite{miao2015,bendory17}.
  This data often excludes a region around the $\mbf{k}$-space origin,
  due to excessive intensity.
  However, since such data are also the Fourier transform of the
  image {\em autocorrelation}, such missing data may be recovered
  by solving a
  %potentially ill-conditioned
  linear system involving a Fourier submatrix \cite{epsteinhole}.
  In the one-dimensional (1D) model, $q \ll N$
  is then the number of missing data,
  and $\al=1-2/m$ where $m>2$ is an oversampling factor.
  The recovered data can help with the subsequent phase retrieval problem;
  this motivated the present study.
%We believe that 2D condition number lower bounds are easily generated
%from products of the 1D case analyzed here.
\een

More widely, linear systems involving rectangular Vandermonde matrices
arise in a variety of signal processing and parameter identification
problems \cite{bazan00,liliao18}.
The square of the singular values of Fourier submatrices
are the eigenvalues controlling the
maximum space- and frequency-concentration of
{\em periodic discrete prolate spheroidal sequences} (P-DPSS)
\cite{grunbaum81,jain81,matthysen16,zhu17},
the fully discrete ($\C^N\to\C^N$) analogues of
prolate spheroidal wavefunctions \cite{slepianI,osipov}.
As Zhu et al.~\cite{zhu17} state, ``there exist comparatively few results
concerning the P-DPSS eigenvalues.''

%*** and The right sing vecs of $A$ are periodic discrete prolate sequences (PDPSS, i.e., eigenfunctions of $A^\ast A$).
%Our work therefore places new upper bounds on the minimum eigenvalue.
%Since linear systems involving the DFT are quite common,
%there are probably many other applications of which we are unaware.

\subsection{Results}  % ssssssssssssssssssssssssssssssssssssssssssssssssssss
\label{s:res}

%almost every submatrix of $F$ is exponentially ill-conditioned with respect to
%growth in its dimensions, in a precise manner depending on the
%scaled dimensions $p/N$ and $q/N$.
%with a rate that depends on the shape $(\al,\bt)$.

Each of our three results is a lower bound on the
condition number of contiguous Fourier submatrices,
that, barring algebraic prefactors, is exponential in
the submatrix size.
Each bound has all constants explicit.
To provide insight, we also present these as $N\to\infty$ asymptotic rates
when the shape parameters $(\al,\bt)$ are held fixed (Table~\ref{t:sum}).
Yet we emphasize that the main theorems are {\em not
  asymptotic results}---in particular, between them they cover all $N$ and,
using the $p\leftrightarrow q$ symmetry of the condition number,
all $p$ and $q$.

We include our first result, even though it will be essentially
superceded by Theorem~\ref{t:kb},
%(which doubles its exponential rate),
because its prefactor is slightly stronger,
and moreover its elementary proof (Section~\ref{s:m}) is instructive.

\begin{thm}  % ttttttttttttttttttttttttttttttttttttttttttttttttttttttttttt
  Let $A$ be a cyclically contiguous $p\times q$ submatrix of the
  $N\times N$ discrete Fourier matrix $F$ given by \eqref{F},
  with $2 < q \le p < N-2$.
  Then the condition number of the matrix $A$ obeys
  \be
  \cond(A)
  \;\ge\;
  \frac{\sqrt{p}(1-\pbar/N)^{1/4}}{6\qbar^{1/4} \sqrt{N}}
  e^{\frac{\pi}{4}(1-\pbar/N)\qbar}~.
  \label{condLB}
  \ee
  where $\qbar$ is the largest even integer smaller than $q$,
  and $\pbar$ is the smallest integer of the same parity as $N$ larger than $p$.
  \label{t:m}
\end{thm}

This applies to  the square or ``tall'' case $q\le p$;
if instead $q>p$ (the submatrix is ``fat''), one applies
this theorem to its Hermitian adjoint $A^\ast$, since
$\cond(A^\ast) = \cond(A)$.
It is thus possible to rephrase the theorem (and the two below) in
a symmetrized form that applies without such conditions on $p$ and $q$.
We will do this in terms of fixed fractional sizes or shape parameters
$\al := p/N$ and $\bt:=q/N$.
Since $q-\qbar$ and $p-\pbar$ are at most 2, they may be absorbed
into a prefactor.
Thus,\footnote{Here we use the $\Omega$ symbol in the lower bound or Knuth sense
  that the right-hand side is $\bigO$ of the left-hand side.}
$$
\cond(A) \; = \; \Omega\bigl( N^{-1/4} e^{\rho(\al,\bt) N} \bigr)
~, \qquad N\to\infty
$$
where the rate $\rho$, and the implied constant,
depend on $\al$ and $\bt$.
Explicitly, $\rho(\al,\bt) = \frac{\pi}{4}\bt(1-\al)$ if $\bt\le\al$.
Combining this with the case $\bt>\al$ by swapping $\al$ and $\bt$,
we get the rate $\rho = \frac{\pi}{4}[\min(\alpha,\beta)- \alpha\beta]$
applying for all $(\al,\bt)\in(0,1)^2$.
These two forms are summarized in the first row of Table~\ref{t:sum}.

Our next result is similar but has a {\em doubled} rate,
and a more involved proof (Section~\ref{s:kb}).
Here $I_0$ will denote the modified Bessel function of order zero
\cite[(10.25.2)]{dlmf}.

\begin{thm} \label{t:kb}   % tttttttttttttttttttttttttttttttttttttttttttttttt
  Let $A$ be a cyclically contiguous $p\times q$ submatrix of the
  $N\times N$ Fourier matrix, with $1\le q \le p < N$.
  Then
  \be
  \cond(A)
  \;\ge\;
  \frac{I_0\bigl( \frac{\pi}{2}(1-p/N)q \bigr)-1}
       {2\bigl(\sqrt{N/p} + 6q\bigr)}
  ~.
  \label{condkb}
  \ee
\end{thm}

Using the asymptotic $I_0(z) \sim e^{z}/\sqrt{2\pi z}$
given in \cite[(10.3.4)]{dlmf}, and keeping only dominant
terms, we get, in terms of the shape parameters,
$$
\cond(A) \; = \; \Omega\bigl( N^{-3/2} e^{\rho(\al,\bt) N} \bigr)
~, \qquad N\to\infty
$$
with an explicit rate $\rho$ precisely double that from
the first theorem, for each $\al$ and $\bt$. This is summarized
in the abstract, and in the second row of Table~\ref{t:sum}.

Our final main result
(proved in Section~\ref{s:corner})
improves upon the above in the case of
small $\al$ and $\bt$, i.e., in the ``corner'' of $(\al,\bt)$ space.

\begin{thm} \label{t:corner} % tttttttttttttttttttttttttttttttttttttttttttttttt
  Let $A$ be a cyclically contiguous $p\times q$ submatrix of the
  $N\times N$ Fourier matrix, with $1< q \le p < 4N / e\pi + 1$.
  Then
  \be
  \cond(A)
  \;\ge\;
 \frac{1- (e\pi(p-1)/4N)}{2\sqrt{q}}
  \left(\frac{4N}{e \pi(p-1)}\right)^{q-1}
  ~.
  \label{condcorner}
  \ee
\end{thm}

When $p,q\gg 1$ it again
makes sense to approximate $p-1$ by $p$, and $q-1$ by $q$,
and rephase this in terms of the shape parameters, giving
$$
\cond(A) \; = \; \Omega\bigl( N^{-1/2} e^{\rho(\al,\bt) N} \bigr)
~, \qquad \al,\bt < 4/e\pi\approx 0.468~, \quad N\to\infty
$$
with $\rho = \bt \log(4/e\pi\al)$ when $\bt\le\al$. Its symmetrized
form is listed on the last row of Table~\ref{t:sum}.

When is Theorem~\ref{t:corner} stronger than Theorem~\ref{t:kb}?
Equating their rates $\bt \log(4/e\pi\al) = \rho = \frac{\pi}{2}\bt(1-\al)$
gives a transcendental equation in $\al$ with solution
$\al_\ast \approx 0.117$.
%fzero(@(a) log(4/(exp(1)*pi*a)) - (pi/2)*(1-a),0.2)
%ans =         0.117019018358476
Including the symmetrized result,
Theorem~\ref{t:corner} is then stronger for all $(\al,\bt)\in(0,\al_\ast)^2$,
i.e., in the corner region occupying about $1.4\%$ of shape space.
(In Section~\ref{s:symm} it is shown how this result also applies
asymptotically to $(1-\al,1-\bt)$, hence also in the diagonally-opposite corner.)

\begin{table}[t] % tttttttttttttttttttttttttttttttttttttttttttttttttttt
  \centering
  \footnotesize
  \begin{tabular}{llll}
    theorem & rate $\rho$ for $\bt\le\al$ & rate $\rho$ in general$^\dag$ & proof technique and section\\
    \hline
    Thm.~\ref{t:m}
    & $\frac{\pi}{4}\bt(1-\al)$
    & $\frac{\pi}{4}[\min(\alpha,\beta)- \alpha\beta]$
    & periodized Gaussian trial (Sec.~\ref{s:m})
    \\
    Thm.~\ref{t:kb}
    & $\frac{\pi}{2}\bt(1-\al)$
    & $\frac{\pi}{2}[\min(\alpha,\beta)- \alpha\beta]$
    & periodized Kaiser--Bessel trial (Sec.~\ref{s:kb})
    \\
    Thm.~\ref{t:corner}
    & $\bt\log\frac{4}{e\pi\al}$
    & $\min(\alpha,\beta)\log \frac{4}{e\pi\max(\alpha,\beta)}$
    & low-rank $e^{ixt}$ kernel approx.\ (Sec.~\ref{s:corner})
    \\
    \hline
  \end{tabular}
  \vspace{0ex}
  \ca{Summary of exponential growth rates $\rho=\rho(\al,\bt)$ of the
    lower bound $e^{\rho N}$ on the condition number of a $p\times q$
    contiguous submatrix of the $N\times N$ Fourier matrix,
    as a function of fixed shape parameters
    $\al:=p/N$ and $\bt:=q/N$, asymptotically as $N\to \infty$.
    We drop algebraic prefactors,
    and drop $\bigO(1)$ changes in $p, q$, since we assume $p,q\gg 1$.
    Each row summarizes a different theorem proven in this paper.
    The second column presents the simpler ``non-fat'' $q\le p$ submatrix case,
    and the third column the general case.
    The $\dag$ is a reminder that Theorem~\ref{t:corner} only applies for
    $\al,\bt < 4/e\pi\approx 0.468$.
  }{t:sum}
\end{table}  % ttttttttttttttttttttttttttttttttttttttttttttttttttttttt

In summary, Table~\ref{t:sum} compares the exponential rates $\rho$
in these three theorems, dropping the algebraic prefactors.
Remark~\ref{r:prolrate} below shows that Theorem~\ref{t:kb}
tends to have a sharp rate at $(1,0)$ and $(0,1)$.
Figure~\ref{f:rates} compares the rates from Theorem~\ref{t:kb}
(see panels (b,e)) and Theorem~\ref{t:corner} (panel (e))
against the {\em empirical} numerical growth rate
of $\cond(A)$, denoted by $\tilde\rho(\al,\bt)$.
In short (see Section~\ref{s:num}),
the lower bound from the stronger of these two theorems
appears to be within a factor of 2 of the empirical rate uniformly
over shape space $(0,1)^2$.

% pppppppppppppppppppppppppppppppppppppppppppppppppppppppppppppppppppppppppppp
\subsection{Relation to prior work}

Here we compare our findings to the few existing
lower bounds on the condition number.
At the end of this section we discuss an asymptotic connection
to eigenvalues of the prolate matrix.
We also note that there have been a couple of small-scale numerical
studies
\cite[Fig.~13]{adcock14}
% they fix q=200 and sweep over 1/4 < alpha < 1,  and alpha/4 < beta < alpha
% (whatever changing N that corresponds to). But q=200 is huge, so cond #
% seems off.
\cite[Table~4]{pan16}
% square case only
\cite[Fig.~4(a)]{liliao18}.
% only goes to q=5 (take A=1 cluster).

The foundational work of Edelman--McCorquodale--Toledo \cite[Thm.~3]{edelman99}
showed that square Fourier
submatrices $A$ with $\alpha=\beta=1/n$, $n=2,3,\dots$,
have an $\epsilon$-rank of $\alpha^2 N$, asymptotically
as $\alpha N\to\infty$.
% ** does their arg work for any p=q< N ?
However, their analysis did not access $\cond(A)$.
Recently Zhu et al.~\cite[Cor.~1]{zhu17}
gave a more refined $\epsilon$ dependence
of the asymptotic distribution of singular values of such an $A$.
%with $\alpha=\beta<1/2$.
%Their result is quite technical,
%bounding the difference between Slepian's semi-discrete prolate
%matrix \cite{slepianV} and its fully discrete
%analog (in our setting $A^\ast A$), and using the Riemann zeta function.
They show that $\sigma_j(A) \le \epsilon$
for $j \ge j_\epsilon$,
with a formula for
$j_\epsilon$ that is at least%
\footnote{Here for simplicity we ignore their nonnegative second term
in $R(L,M,\epsilon)$.}
$\alpha^2 N + (\frac{4}{\pi^2} \log 8\alpha N + 6)\log(16/\epsilon^2)$.
To convert this to a lower bound on condition number
one equates $j_\epsilon$ to $\alpha N$, the submatrix size, then
solves for $\epsilon$.
Simplifying somewhat, the resulting lower bound cannot be stronger than
$$
\cond(A) \sim \epsilon^{-1} \gtrsim \exp \biggl[
\frac{\alpha(1-\alpha)}{(8/\pi^2)\log 8\alpha N + 12} N \biggr]~,
$$
which is super-algebraic, but falls short of $e^{\rho N}$ for any positive $\rho$.

To our knowledge, the chief prior exponential lower bounds on $\cond(A)$
are the following three.

1) Pan \cite[Thm.~6.2]{pan16} proves that for $p=q=N/2$,
$\cond(A) \ge \sqrt{N}2^{N/4 -1}$, \ie $\rho = (\log 2)/4 \approx 0.173$.
% ...exactly twice Moitra's rate for this shape.
We see that both Theorems~\ref{t:m} and \ref{t:kb} are stronger than
this, the latter giving $\rho = \pi/8 \approx 0.393$ at $\al=\bt=1/2$.

2) Moitra \cite[Thm.~3.1]{moitra} gives the only proof of which we are aware
that the condition number of a submatrix of {\em general} shape
grows at least exponentially, although no rate is given.
%and an asymptotic condition $q =\Omega(\log N)$ is needed.
His method is similar to that of our Theorem~\ref{t:kb}---although we found it independently---but with trial vector $\mbf{v}$ chosen
as a high power of the Fej\'er kernel.
By tracking the rate in his proof%
\footnote{Note that his short proof is unclear about whether the width
  $\ell$ and power $r$ may be non-integer-valued, which would
  be needed to make claims for almost all $p$ and $q$, as we do.}
we get, in our notation,
$$
\rho_\tbox{Moitra} \;=\; (\log \sqrt{2}) \bt(1-\al)~, \qquad \bt\le\al~.
$$
This has a similar form as our first two theorems (see Table~\ref{t:sum}),
but Theorem~\ref{t:kb}
improves upon it by a factor of about 4.5, for all $(\al,\bt)$.
%but since $\log \sqrt{2} \approx 0.346$, whereas $\pi/2 \approx 1.57$, our

3)
%Due to their relevance for discrete-data super-resolution with an equispaced clump of $q$ sources,
In the super-resolution literature,
lower bounds on the condition number have been proven in the case of fixed $q$,
and $\al\ll 1$. They take the general form
$\cond(A) \ge (c\al)^{-q+1} \propto e^{\bt N \log (1/c\al)}$,
similar to the last row of our Table~\ref{t:sum}.
The strongest such result that we know of
is that of Li--Liao \cite{liliao18} (see also \cite[Ex.~5.1]{kunis19}).
In our notation, \cite[Prop.~3]{liliao18} states
%\footnote{their $\lambda$ is our $q$; their $M+1$ is our $p$; their $\alpha:=\Delta M$ tends to ours as $M\to\infty$.}
%(noting that their $\al$ matches ours),
\be
\sigma_\tbox{min}(A) \;\le\;
\biggl(\begin{array}{c}2q-2 \\ q-1\end{array}\biggr)^{-1/2}
  2\sqrt{p} (2\pi\alpha)^{q-1}
  \;\le\; \sqrt{8pq}(\pi\alpha)^{q-1} ~,
  \label{liliao}
\ee
  where in the second form we bounded the central binomial coefficient.
  Thus their constant is $c=\pi$.
  Their proof exploits the $q$th-order finite difference trial vector
  %$v_j = (-1)^j \left(\begin{array}{c}q-1 \\ j-1\end{array}\right)$,
  $v_j = (-1)^j (q-1)!/(q-j)!(j-1)!$,
  $j=1,\dots,q$, whose first $q$ moments vanish,
  following Donoho \cite[\S 7.4]{donoho92}. % he was for continuous data model
  The result \eqref{liliao} has restrictive conditions:
  $\al \le 1/(C(q)\sqrt{p})$, where one may check that
  $C(q) \sim 4^q$. Thus \eqref{liliao}
  does not apply for fixed $q$ and $\al$ as $N\to\infty$.
  In constrast, our Theorem~\ref{t:corner}
  applies to all submatrix sizes up to $0.468 N$,
  and furthermore its rate constant $c$ is $4/e\approx 1.47$ times stronger.

\begin{rmk}[The prolate matrix and sharpness]\label{r:prolrate} % rrrrrrrrrrrrr
  There is a limit in which the singular values of $A$ are already
  well understood \cite{edelman99,adcock14,zhu17,batenkov20}.
  When $N\to\infty$ with $\al$ and $q$ held constant (so that
  the height of $A$ is much larger than its width),
  a Riemann sum shows that $A^\ast A$ tends to a multiple of Slepian's
  {\em prolate matrix} \cite{slepianV,varah93}.
  The latter is $P(q,\al/2)$ in standard
  notation, with elements $(P(q,\al/2))_{jk} = \al \sinc (\pi \al(j-k))$,
  $j,k =1,\dots,q$,
  recalling that $\sinc x := (\sin x)/x$ for $x\neq0$ and 1 otherwise.
  Thus, in this limit, the
  singular values of $A$ are the square-roots of the eigenvalues
  of $P(q,\al/2)$. Slepian proved the $q\to\infty$ asymptotic
  for the smallest such eigenvalue \cite[Eqs.~(13), (58)]{slepianV},
  $$
  \lambda_0\big( P(q,\al/2) \big) \;=\;
  1-\lambda_{q-1}\big( P(q,(1-\al)/2)\big)
  \;\sim\;
  2^{9/4}\sqrt{\pi q}\frac{(1-\cos \pi \al)^{1/4}}{(1+\cos \pi \al)^{1/2}}
  \,\theta^{-q}
  $$
  where $\theta = \cot^2(\pi\al/4) \in (1,\infty)$ \cite[\S 3.2]{adcock14}.
  Recalling that $q=\beta N$, the term $\theta^{-q}$ predicts an exponential
  growth rate for $\cond(A)$ of
  \be
  \rho_\tbox{\rm prolate} \;:=\; \bt \log \cot \frac{\pi\al}{4}~,
  \qquad \bt\ll\al~,
  \label{prolrate}
  \ee
  which should be compared with our results in the second column of
  Table~\ref{t:sum}.

  In particular, the first term in the Taylor expansion of \eqref{prolrate}
  about $\al=1$ is precisely our $\frac{\pi}{2}\bt(1-\al)$, proving that
  as $\al\to1^-$ and $\bt\ll 1$, the rate in Theorem~\ref{t:kb} is sharp.
  Instead taking $\al\to0^+$ gives $\bt \log(4/\pi\al)$, indicating
  that the rate of Theorem~\ref{t:corner} fails to be sharp by an amount $\bt$.
  We have also checked that for $\bt\ll\al$ the empirical rates
  shown in Figure~\ref{f:rates}(a) match \eqref{prolrate} very well:
  the agreement is within 0.05 uniformly over $0<\al<1$, $0<\bt<\al/3$.
\end{rmk}   % rrrrrrrrrrrrrrrrrrrrrrrrrrrrrrrrrrrrrrrrrrrrrrrrrrrrrrrrrrrr

We emphasize that,
in contrast to the above prolate asymptotics which fix $q$ and $\al$ with $N \to\infty$,
our Theorems~\ref{t:m}--\ref{t:corner} are {\em not} asymptotic: they apply to arbitrary $N$ and essentially arbitrary $p$ and $q$.

% oooooooooooooooooooooooooooooooooooooooooooooooooooooooooooooooooooooooooooo
\subsection{Overview of proof methods}
\label{s:pf}

We now outline the tools used to prove the three main theorems.
The definition of matrix condition number is
\be
\cond(A) := \frac{\sigma_1(A)}{\sigma_\tbox{min}(A)}
~,
\label{conddef}
\ee
where $\sigma_1 \ge \sigma_2 \ge \dots \ge
\sigma_{\min(p,q)} =: \sigma_\tbox{min}$ are the singular values of $A$.
All three proofs place exponentially small upper bounds on $\sigma_\tbox{min}(A)$,
then combine this with the following simple bound.
\begin{pro}\label{p:s1}   % ppppppppppppppppppppppppppppp
  The largest singular value of any $p\times q$ submatrix of the Fourier
  matrix \eqref{F} obeys
  \be
  \sigma_1(A) \ge \sqrt{p}~.
  \ee
\end{pro}
\begin{proof}
  The operator norm of $A$ is bounded by the Frobenius norm
  \cite[(2.3.7)]{golubvanloan} giving,
  $\sqrt{pq} = \|A\|_F \le \sqrt{q}\|A\|$.
  The result follows since $\sigma_1(A) = \|A\|$.
\end{proof}

Theorems~\ref{t:m} and \ref{t:kb} will use the variational bound on
$\sigma_\tbox{min}$.
Namely, if $q\le p$, then
\be
\sigma_\tbox{min}(A) = \sigma_q \le \frac{\|A\mbf{v}\|_2}{\|\mbf{v}\|_2}~,
\qquad\mbox{ for any } \mbf{v}\in\C^q,  \mbf{v}\neq\mbf{0}~.
\label{mingro}
\ee
If instead $q>p$ (a ``fat'' submatrix), \eqref{mingro} no longer holds,
explaining the hypothesis $q\le p$ in Theorems~\ref{t:m} and \ref{t:kb}.

\bfi[t] % ffffffffffffffffffffffffffffffffffffffffffffffffffffffffffffffffffffff
\centering
\ig{width=5.2in}{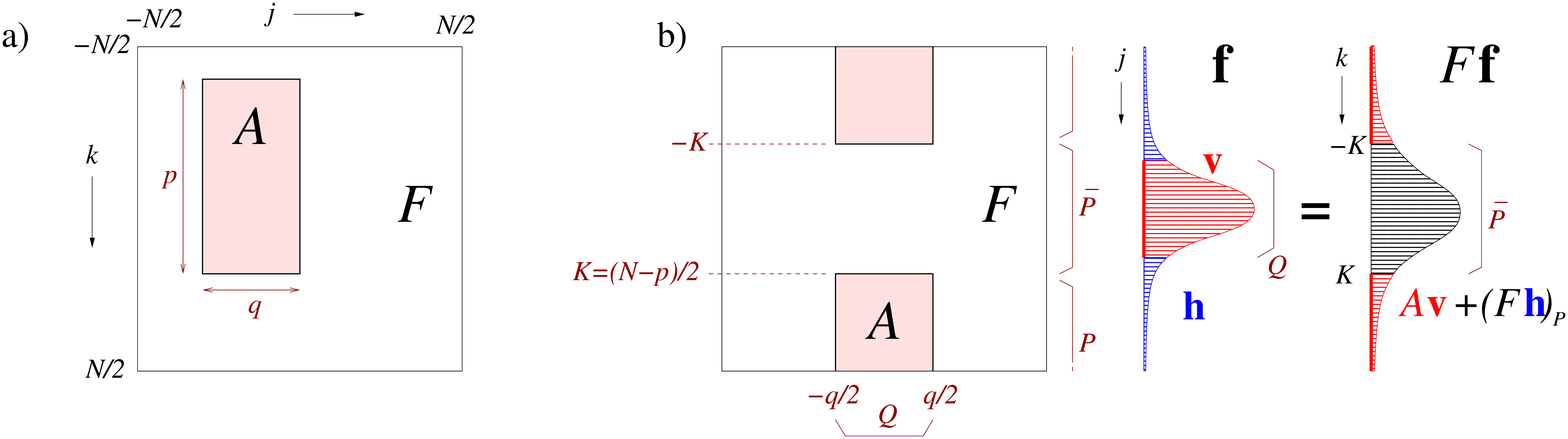}
\ca{The main proof idea for Theorems~\ref{t:m} and \ref{t:kb}.
  a) shows the submatrix $A$ (pink) within the $N\times N$ Fourier matrix $F$.
  b) shows the cyclically shifted $A$ (pink) again within $F$,
  and the formation of the DFT matrix-vector product $F\mbf{f}$.
  The vector $\mbf{f}$ is exponentially small outside of the central
  index set $Q$, and has a known DFT which is
  exponentially small in the high frequency index set $P$.
  The length-$q$ vector (shown in red) is $\mbf{v} = \mbf{f}|_Q$, i.e.,
  restricted to the ``input'' indices.
  $F\mbf{f}$ in the ``output'' index set $P$ equals $A\mbf{v}$
  plus a correction $(F\mbf{h})|_P$, where $\mbf{h}$ (shown blue)
  is the part of $\mbf{f}$ outside of $Q$.
  For Theorem~\ref{t:m}, $\mbf{f}$ samples a 
  periodized Gaussian (as sketched); for
  Theorem~\ref{t:kb} it samples a periodized ``deplinthed'' Kaiser--Bessel
  function (see Figure~\ref{f:dkb}; in this case $\mbf{h}=\mbf{0}$).
}{f:idea}
\efi

Our trick to construct a trial vector $\mbf{v}$ for which $A\mbf{v}$ is
nearly (or exactly) known is by {\em embedding} this matrix-vector product
in the larger one, $F\mbf{f}$, for some $\mbf{f}\in\C^N$.
Firstly we note that cyclic horizontal or vertical
translations of the submatrix location within
$F$ do not affect its singular values. This can
be seen by noting that a cyclic right-translation by $m$ of a submatrix $A$
is equivalent to left-multiplication of $A$ by the
diagonal unitary matrix with diagonal entries $\{e^{2\pi i j m}\}_{j\in J}$,
where $J$ is the set of row indices of $A$
(this has been pointed out, e.g., in \cite[Sec.~III.C]{zhu17}).
Thus we translate the index sets of $A$ to be the $p$ highest magnitude
output frequencies (see $P$ defined by \eqref{P}), and the $q$ lowest-magnitude
inputs $Q := \{j\in\Z: -q/2\le j < q/2\}$; see Figure~\ref{f:idea}.
The trial vector $\mbf{v}$ will then be $\mbf{f}|_Q$, the restriction
of $\mbf{f}$ to the central index set $Q$.
The desired $A\mbf{v}$ is then $(F\mbf{f})|_P$, plus a
correction of size at most the norm of $\mbf{f}$ outside $Q$, which
is arranged to be small or zero.

Thus we seek an $\mbf{f}$ that is exponentially small
outside of $Q$, whose DFT, $F\mbf{f}$, is known and exponentially small
in $P$.
We build this from a known Fourier transform pair $f(t)$ and
$\hat f(\omega)$ on $\R$, by applying the
phased variant of the Poisson summation formula
\cite[Thm.~3.1]{steinfourier}
\be
\sum_{j\in\Z} e^{2\pi i j \omega} f(j) = \sum_{m\in\Z} \hat f(\omega+m)~.
\label{PSF}
\ee
Setting $\omega=k/N$ makes the left-hand side equal to the $k$th entry of
$F\mbf{f}$, for a vector $\mbf{f}$ whose $j$th entry
is $f_j = \sum_{n\in\Z} f(j+nN)$,
a sample of the $N$-periodization of the original function $f$.

In Theorem~\ref{t:m} we choose the Gaussian
Fourier pair, which is of course well localized
both in position and frequency space.
It also has monotonic tails, allowing sums to be bounded by simple integrals.
However, the problem of {\em optimal} (in the $L^2(\R)$ sense)
simultaneous position- and frequency-localization was solved
by Slepian and co-workers in the form of the
prolate spheroidal wavefunctions (PSWFs) of order zero,
in particular the first such function $\psi_0$
\cite{slepianI,slepianIII,osipov}.
Thus one might hope that by choosing $f$
as a scaled truncated $\psi_0$,
making $\hat f$ a scaled (but not truncated) $\psi_0$,
one could beat the Gaussian rate.
Yet, we were unable to find estimates on the tails of $\psi_0$
that enable bounding the right-hand side sum in \eqref{PSF}.
Instead, Theorem~\ref{t:kb} relies on the so-called Kaiser--Bessel (KB)
Fourier transform pair \cite{kaiser,fourmont}
defined by \eqref{KBpair},
which has the same optimal exponential rate of localization as the PSWF
(see, e.g., \cite[Sec.~5]{keranal}).
The non-compactly-supported member of this pair
involves only sinc and other elementary functions,
allowing a bound on its tail sum.
Since the tail is {\em oscillatory} rather than monotonic
(see Figure~\ref{f:dkb}),
the bound is involved, although elementary.
Here we draw on the thesis of Fourmont
\cite{fourmontthesis};
he was concerned with error of the nonuniform
FFT \cite{fourmont,usingnfft,keranal} where the criteria for
a good spreading kernel are similar to those for $f$.
We suspect that this work is the first to exploit the KB pair
as an {\em analysis} (as opposed to numerical) tool.
%KB scarcely known outside sig proc / imaging.

In contrast, to prove Theorem~\ref{t:corner},
$\sigma_\tbox{min}$ is bounded from above
via the SVD rank approximation theorem\footnote{This is often referred to as the 
Ekhart--Young theorem, although it is due to Schmidt.}
\cite[Thm.~2.5.3]{golubvanloan}.
Namely, $\sigma_q$ is bounded from above
by the error of any rank-$(q-1)$ approximation of $A$, in the operator norm.
A rapidly-convergent approximation comes by
sampling a carefully chosen continuous kernel expansion of the form
$$
e^{ixt} = \sum_{n=1}^\infty f_n(x) g_n(t)
$$
on regular grids in $x$ and $t$, so that its samples give the elements of $A$.
This sampling idea has been used by O'Neil--Rokhlin to
bound the numerical rank of $A$ \cite[Cor.~3.4]{oneil07}.

The structure of the rest of this paper is as follows.
The next three sections correspond to the three main theorems:
Section~\ref{s:m} proves the elementary Gaussian rate,
Section~\ref{s:kb} proves the doubled rate based on the Kaiser--Bessel pair,
then Section~\ref{s:corner} proves the further improved rate in the
corner region of the $(\alpha,\beta)$ plane.
Finally, Section~\ref{s:disc}
has a detailed numerical comparison to the empirical growth rate,
discusses symmetries (both exact and near) in the $(\al,\bt)$ plane,
and draws some conclusions.
A short Appendix proves the Kaiser--Bessel transform pair.

% PPPPPPPPPPPPPPPPPPPPPPPPPPPPPPPPPPPPPPPPPPPPPPPPPPPPPPPPPPPPPPPPPPPPPPPPPPPP
\section{Proof of elementary Gaussian rate (Theorem~\ref{t:m})}
\label{s:m}

We start with a simple identity that the DFT of a
periodized discrete Gaussian is another periodized discrete Gaussian.

\begin{pro}  % ppppppppppppppppppppppppppppppppppppppppppppppppppppppppppppppp
  Let $N\in\N$, $\sigma>0$, and
  define the $N$-periodic vector $\mbf{f}\in\C^N$ by its entries
  \be
  f_j = \sum_{n\in\Z} e^{-\half (j+nN)^2/\sigma^2}~,
  \qquad -N/2\le j < N/2~.
  \label{fj}
 \ee
 Then the $k$th component of the DFT of $\mbf{f}$ is
 \be
 (F\mbf{f})_k = \sqrt{2\pi} \sigma \sum_{m\in\Z} 
 e^{-2(\pi \sigma/N)^2(k+mN)^2}~,
 \qquad   -N/2\le k < N/2~.
 \label{Ff}
 \ee
\end{pro}
\begin{proof}
  With the (``number theorist'') Fourier transform convention
  \be
  \hat f(\omega) := \int_\R f(t) e^{2\pi i \omega t} dt~,
  \qquad \omega\in\R
  \label{ft}
  \ee
  we have the usual continuous Fourier transform pair
  \be
  f(t) = e^{-\half (t/\sigma)^2}   \qquad \longleftrightarrow \qquad
  \hat f(\omega) = \sqrt{2\pi} \sigma e^{-2(\pi\sigma\omega)^2}~.
  \label{pair}
  \ee
  Inserting this $f$ into the Poisson summation formula \eqref{PSF}
  and setting $\omega = k/N$ gives the discrete sums
  $$
  \sum_{j\in\Z} e^{2\pi i jk/N} e^{-\half (j/\sigma)^2} =
  \sqrt{2\pi} \sigma \sum_{m\in\Z} e^{-2(\pi\sigma)^2 (k/N + m)^2} ~,
  \qquad \forall k \in \N~.
  $$
  By grouping terms with the same index modulo $N$,
  and using \eqref{F} and \eqref{fj},
  the left-hand side can be seen to be $(F\mbf{f})_k$.  
\end{proof}

\begin{pro}   % pppppppppppppppppppppppppppppppppppppppppppppppppppppppppppppppp
  \eqref{fj} and \eqref{Ff} may be bounded uniformly over
  their centered index ranges
  by (non-periodic) Gaussians, namely
  \be
  f_j \;\le\; \left(2+ \sqrt{2\pi}\frac{\sigma}{N}\right) e^{-\half (j/\sigma)^2}~,
    \qquad -N/2\le j < N/2~.
  \label{fjub}
  \ee
  \be
  (F\mbf{f})_k \;\le\; (\sqrt{8\pi}\sigma + 1)
  e^{-2 (\pi\sigma k/N)^2}~,
  \qquad -N/2\le k < N/2~.
  \label{Ffub}
  \ee
\end{pro}
\begin{proof}
  Assume $j\ge 0$, then we may drop the cross-term in the square,
  then exploit monotonicity to
  bound a Gaussian sum by its integral, to get
  \bea
  \sum_{n\ge 0} e^{-\half (j+nN)^2/\sigma^2} &\le&
  e^{-\half (j/\sigma)^2} \sum_{n\ge 0} e^{-\half (nN/\sigma)^2}
  \nonumber \\
  &\le&
  e^{-\half (j/\sigma)^2} \left(1 + \int_0^\infty e^{-\half (tN/\sigma)^2} dt \right)
  \le
  e^{-\half (j/\sigma)^2} \left(1 + \sqrt{\frac{\pi}{2}}\frac{\sigma}{N} \right)~.
  \nonumber
  \eea
  The negative terms are handled by shifting the sum
  $$
  \sum_{n<0} e^{-\half (j+nN)^2/\sigma^2} = \sum_{n\ge0} e^{-\half (N-j + nN)^2/\sigma^2}~,
  $$
  which, since $j\le N/2$, is termwise no larger than the non-negative sum,
  giving an overall factor of two, hence \eqref{fjub}.
  The $j<0$ case is handled by noting $f_{-j} = f_j$.
  \eqref{Ffub} follows by an identical method.
\end{proof}

\begin{lem}   % llllllllllllllllllllllllllllllllllllllllllllllllllllllllll
  Let $A$ be a cyclically contiguous $p\times q$ submatrix of the
  $N\times N$ discrete Fourier matrix $F$, with $p\ge q$ (i.e.,
  the submatrix is either square or ``tall''), $q>2$, and $p<N-2$.
  Then, in terms of $\pbar$ and $\qbar$
  in Theorem~\ref{t:m}, the smallest singular value of $A$ obeys
  \be
  \sigma_\tbox{min}(A)
  \;\le\; 6 \biggl(\frac{\qbar}{1-\pbar/N}\biggr)^{1/4} \sqrt{N}
  \cdot e^{-\frac{\pi}{4} \qbar(1-\pbar/N)}~.
  \label{sigUB}
  \ee
\label{l:m}
\end{lem}

\begin{proof}
  As discussed in Section~\ref{s:pf},
  we translate $A$ to be centered vertically about the highest
  frequency and horizontally about the lowest; see Figure~\ref{f:idea}(b).
  We then apply \eqref{mingro}, with $\mbf{v}$ chosen as follows.
  We select the central $q$ elements of $\mbf{f}$,
  given by \eqref{fj}, with width parameter $\sigma>0$ to be specified later,
  i.e.,
  \be
  \mbf{v} := \{f_j\}_{-q/2\le j < q/2}~.
  \label{vf}
  \ee
  Let $\mbf{g}\in\C^N$ be $\mbf{v}$ embedded in the vector of length $N$,
  i.e.\ with elements $g_j = f_j$ for $-q/2\le j < q/2$,
  zero otherwise.
  The remainder we denote by $\mbf{h}\in\C^N$, that is,
  $h_j = 0$ for all $-q/2\le j < q/2$, and $h_j=f_j$ otherwise.
  In summary,
  $\mbf{f}=\mbf{g} + \mbf{h}$,
  with $\mbf{g}$ supported only in the central region $Q$
  while $\mbf{h}$ is supported only in its complement; see Figure~\ref{f:idea}(b).
  %Thus $F\mbf{f}=F\mbf{g} + F\mbf{h}$
  Then, letting $P$ be the (output) index set
  \be
  P := \{-N/2\le k < -(N-p)/2\} \cup \{(N-p)/2\le k< N/2\}~,
  \label{P}
  \ee
  and using the inequality $(a-b)^2 < 2(a^2+b^2)$, we bound
  \bea
  \|A\mbf{v}\|_2^2
  &=&
  \sum_{k\in P} |(F\mbf{g})_k|^2 = \sum_{k\in P} |(F\mbf{f})_k - (F\mbf{h})_k|^2
  \nonumber \\
  &\le& 2 \sum_{k\in P} |(F\mbf{f})_k|^2  + 2 \|F\mbf{h}\|_2^2
  \le 2 \sum_{k\in P} |(F\mbf{f})_k|^2  + 2 \|F\|^2\|\mbf{h}\|_2^2
  ~.
  \label{Aubnd}
  \eea
  To bound the first sum in \eqref{Aubnd}
  we set $K:=(N-\pbar)/2$ as
  a lower bound on the half-width of $\overline{P}$,
  the complement of $P$ in the full set $\{-N/2 \le k < N/2\}$.
  Using this and \eqref{Ffub},
%applying
%  $(\sum a_k)^2 \le \sum a_k^2$ when $a_k\ge0$ $\forall k$,  GARBAGE
%  unrolling the sum over $m$, and bounding
%  the tail sum by a Gaussian integral,
  \bea
  \sum_{k\in P} |(F\mbf{f})_k|^2 &\le&
  \sum_{|k|>K, -N/2\le k< N/2} \!\! |(F\mbf{f})_k|^2
  \;\le\;
   2(\sqrt{8\pi}\sigma + 1)^2
  \sum_{k=K+1}^\infty e^{-(2\pi\sigma k/N)^2}
  \nonumber \\
  &\le\;&
  2(\sqrt{8\pi}\sigma + 1)^2
  \sum_{\ell>0}e^{-\left(\frac{2\pi \sigma}{N}\right)^2(K^2+\ell^2+2K\ell)}
  \nonumber \\
  &\le\;&
  2(\sqrt{8\pi}\sigma + 1)^2
    e^{-\left(\frac{2\pi \sigma K}{N}\right)^2}
    \sum_{\ell>0}e^{-\left(\frac{2\pi \sigma \ell}{N}\right)^2}
  \nonumber \\
  &\le&
  2(\sqrt{8\pi}\sigma \!+\!1)^2
  e^{-\left(\frac{2\pi \sigma K}{N}\right)^2}
  \!
  \int_0^\infty
  \!\!
  e^{-\left(\frac{2\pi \sigma x}{N}\right)^2} dx
  =
  4\sqrt{\pi} \sigma N \biggl(1 \!+\! \frac{1}{\sqrt{8\pi}\sigma}\biggr)^2
  \!\!
  e^{-\left(\frac{2\pi \sigma K}{N}\right)^2}
  \!
  .  \nonumber
  \eea
  To bound the second term in \eqref{Aubnd}, we use \eqref{fjub}
  and apply an identical method to get
  \bea
  \|\mbf{h}\|_2^2 &=& \sum_{j\notin Q} f_j^2
  \le \!\!\!  \sum_{|j|>\qbar/2, -N/2\le j < N/2} \!\!\!  f_j^2
  \;\le\;
  2\left(2+ \sqrt{2\pi}\frac{\sigma}{N}\right)^2
  \sum_{j=\qbar/2+1}^\infty
  e^{-(j/\sigma)^2}
  \nonumber\\
  &\le&
  4\sqrt{\pi}\sigma \left(1+ \sqrt{\frac{\pi}{2}}\frac{\sigma}{N}\right)^2
  e^{-(\qbar/2\sigma)^2}~.
  \nonumber
  \eea
  Substituting the last two results and $\|F\|^2=N$ into \eqref{Aubnd}
  and gathering common terms gives
  \be
  \|A\mbf{v}\|_2^2
  \;\le\;
  8\sqrt{\pi} N \sigma
  \left[
    \biggl(1 + \frac{1}{\sqrt{8\pi}\sigma}\biggr)^2
    e^{-(2\pi \sigma K/N)^2}
    +
    \left(1+ \sqrt{\frac{\pi}{2}}\frac{\sigma}{N}\right)^2
    e^{-\qbar^2/4\sigma^2}
    \right]
  \label{Aubnd2}
  \ee
  Choosing the width $\sigma$ to balance the two exponential rates gives
  \be
  \sigma^2 = \frac{\qbar}{2\pi(1-\pbar/N)}~.
  \label{sigbal}
  \ee
  We can now bound the prefactor $(1+1/\sqrt{8\pi}\sigma)^2 \le
  (1+\sqrt{1-\pbar/N}/2\sqrt{\qbar})^2 \le 3$ using $\qbar\ge1$
  implied by the hypothesis $q>2$.
  Also, by the arithmetic-geometric inequality,
  $(1+ \sqrt{\pi}\sigma/\sqrt{2}N)^2 \le
  2 + \pi\sigma^2/N^2 \le 3$ since $N-\pbar \ge 1$ from the hypothesis
  $p<N-2$.
  Substituting these and \eqref{sigbal} into
  \eqref{Aubnd2} gives
  \be
  \|A\mbf{v}\|_2^2
  \;\le\;
  24\sqrt{2} \sqrt{\frac{\qbar}{1-\pbar/N}} N e^{-\frac{\pi}{2} \qbar(1-\pbar/N)}~.
  \label{Aubnd3}
  \ee
  Finally, the crude lower bound $\|\mbf{v}\|^2_2\ge 1$ results
  by keeping only the term $j=n=0$ in the definition \eqref{fj} of $f_j$.
  Combining this with the square-root of \eqref{Aubnd3},
  and simplifying $\sqrt{24\sqrt{2}}<6$, gives \eqref{sigUB}.
\end{proof}

\begin{rmk}
  A similar proof idea---an explicit trial function
  (in their case a truncated Gaussian)
  that is near a function whose Fourier transform
  nearly has compact support---was used by Landau and Pollak in \cite[Lem.~4]{slepianIII}
  to show that the prolate eigenvalue $\mu_0$ is exponentially close to $1$.
\end{rmk}
%also Donoho '92, finding an explicit trial vector to upper bound sigma-min.

With $\sigma_\tbox{min}$ upper bounded by Lemma~\ref{l:m},
and $\sigma_1 \ge \sqrt{p}$ by Proposition~\ref{p:s1},
Theorem~\ref{t:m} follows from inserting these two bounds into
the definition of $\cond(A)$, namely \eqref{conddef}.

% BBBBBBBBBBBBBBBBBBBBBBBBBBBBBBBBBBBBBBBBBBBBBBBBBBBBBBBBBBBBBBBBBBBBBBBBBBB
\section{Proof of Kaiser--Bessel doubled rate (Theorem~\ref{t:kb})}
\label{s:kb}

Theorem~\ref{t:kb} is an immediate consequence of 
the lower bound on $\sigma_1(A)$ (Proposition~\ref{p:s1}),
and the following upper bound on $\sigma_\tbox{min}(A)$.

\begin{thm}\label{t:kbsig} % tttttttttttttttttttttttttttttttttttttttttttttttt
  Let $A$ be a cyclically contiguous $p\times q$ submatrix of the
  $N\times N$ discrete Fourier matrix $F$, with $1\le q\le p < N$ (i.e.,
  either square or ``tall''). Then, expressed in terms of $\al := p/N \in(0,1)$,
  \be
  \sigma_\tbox{min}(A)
  \;\le\; \frac{2\sqrt{N} (1 + 6 \sqrt{\al} q)}{I_0\big(\frac{\pi}{2}(1-\al)q\big) - 1}
  ~,
  \label{kbsig}
  \ee
  where $I_0$ is the modified Bessel function of the first kind of order zero.
\end{thm} % tttttttttttttttttttttttttttttttttttttt

Before proving this, we introduce the main tool.
The Kaiser--Bessel analytic Fourier transform pair is, given a parameter
$\si>0$,
\be
\phi(t) = \left\{
\begin{array}{ll} I_0(\si\sqrt{1-t^2}), & |t|\le 1\\
  0~,& \mbox{otherwise,}\end{array}\right.
\longleftrightarrow \;\;
\hat\phi(\omega) = 2 \sinc \sqrt{(2\pi\omega)^2 - \si^2}
~,
\label{KBpair}
\ee
with the Fourier convention \eqref{ft}.
Apparently due to B.~F.~Logan \cite{kaiserinterview},
it is stated without proof in \cite[p.~232--233]{kaiser},
and since has become popular for windowing and gridding in signal processing
(see \cite{fourmont,usingnfft,keranal} and references within).
Since it is not listed in standard tables,
and we know of no published proof, we include one in the Appendix.
The parameter $\sigma$ may be interpreted as
a ``cut-off'' frequency:
for $2\pi|\omega|<\sigma$ the sinc (and hence sin) has
imaginary argument so is exponentially large (around $e^{\si}$),
whereas for all $2\pi|\omega|\ge\sigma$ it is bounded by $2$.

Our precise choice of Fourier pair will be informed by the need to
bound an infinite algebraic sum \eqref{PSF} over its frequency argument.
Since $\phi$ in \eqref{KBpair} is discontinuous at $\pm 1$
(see Figure~\ref{f:dkb}(a)),
the tails of $\hat\phi$ have sinc-type oscillations whose amplitude decays
only as $|\omega|^{-1}$ (see Figure~\ref{f:dkb}(d)); thus
such algebraic sums are not absolutely convergent,
making the analysis tricky.

\begin{rmk}\label{r:offset}  % rrrrrrrrrrrrrrrrrrrrrrrrrrrrrrrrrrrrrrr
A very similar type of sum over the tails of
the function $\hat\phi$ in \eqref{KBpair}
has already been bounded in a detailed analysis
by Fourmont \cite{fourmontthesis,fourmont}.
However, he relies on the fact that his sum is over the tail
of an algebraic series with {\em zero offset},
allowing him to build on the uniformly bounded Fourier series
$\sum_{k>0} (\sin kx)/k = (\pi-x)/2$ for $0<x<2\pi$, zero for $x=2\pi$.
When series with arbitrary offset
are instead allowed, as we will require, logarithmic
divergence is possible,
as shown by $\sum_{k>0} \sin(2\pi k + \pi/2)/k = \infty$.
\end{rmk}

For this reason, in order to simplify the
analysis we subtract a top-hat function
of height 1 and width 2 from $\phi$ in \eqref{KBpair} to give the
``deplinthed'' Kaiser--Bessel
function, which, in contrast to $\phi$, is continuous on $\R$.
Since this subtraction causes a change in $\hat\phi$ of at most 2,
it preserves its excellent Fourier localization.
Now rescaling $t$ so that the support falls
within a $q$-sized interval, the deplinthed KB pair becomes
\be
f(t) =
\left\{\!
\begin{array}{ll} I_0\biggl(\si\sqrt{1 - \big(\frac{2t}{q}\big)^2}\biggr) - 1, & |t|\le q/2\\
  0~,& \mbox{otherwise,}\end{array}\right.
   %\longleftrightarrow \;
   \hat f(\omega) = q \!\left[
     \sinc \!\sqrt{(\pi q \omega)^2 \!-\! \si^2}
     -
     \sinc \pi q \omega
     \right]
\label{DKBpair}
\ee
which will play the role of \eqref{pair}
in the following proof of Theorem~\ref{t:kbsig}.
This pair is shown by Figure~\ref{f:dkb}(a,b,e).

\begin{proof}
  We apply Poisson summation \eqref{PSF} to the pair
  \eqref{DKBpair}, and set $\omega=k/N$,
to give an explicit formula for the action of the DFT
on $\mbf{f}$, the vector with entries $f_j = f(j)$,
$-N/2\le j < N/2$, being the discrete samples of \eqref{DKBpair}.
This action is
\be
(F\mbf{f})_k =
q\! \sum_{m\in\Z} \! \left[ \sinc \!\sqrt{(\pi q)^2(k/N + m)^2 - \si^2}
  -\sinc \pi q(k/N + m)
  \right], \;\; -N/2\le k < N/2.
\label{DKBFf}
\ee
As in the proof of Lemma~\ref{t:m}, the submatrix $A$ is now
arranged to sit within $F$ so that its input (column) index set is
$-q/2\le j < q/2$, and the output (row) index set $P$ is as in \eqref{P}.

However, in contrast to the Gaussian case, $\si$ may now be chosen up front,
as follows.
We need to ensure that for each index $k\in P$,
all of the arguments in the terms in \eqref{DKBFf} fall at or beyond cutoff
(see Figure~\ref{f:dkb}(b)),
so that their contribution is exponentially small relative to $\|\mbf{f}\|$.
It is sufficient to check this for the term $m=0$, because $|k|\le N/2$.
This gives the criterion $\pi q |k|/N \ge \si$ for all $k\in P$.
For all $k\in P$ we have $|k|\ge(N-p)/2 = (1-\al)N/2$, so the choice
\be
\si = \frac{\pi}{2} (1-\al) q
\label{sig}
\ee
optimally satisfies the criterion, and we fix this from now on.

In order to bound \eqref{DKBFf} we will break up the sum into two
one-sided sums, each of which has the form
\be
S_\si(a,b) \;:=\; \sum_{m\ge 0} \sinc \sqrt{(a m+b)^2-\si^2} - \sinc (a m+b)
~,
\label{S}
\ee
where $a=\pi q$ is the spacing and $b$ the offset
for the arithmetic progression of frequencies $am+b$.
Using the even symmetry of the sinc function,
we now split the sum \eqref{DKBFf} into three parts,
$m<0$, $m>0$, and $m=0$, to get, in terms of \eqref{S}, the three terms
\bea
(F\mbf{f})_k &\;= \;&q S_\si\left(\pi q,\pi q \bigl(1- \frac{k}{N}\bigr)\right)
\;+\; q S_\si\left(\pi q,\pi q \bigl(1+ \frac{k}{N}\bigr) \right)
\nonumber\\
&& + \; q\left[\sinc \sqrt{(\pi q k/N)^2 - \si^2}
- \sinc (\pi q k/N)\right]~. \nonumber
\eea
Here in both instances of the one-sided sum $S_\si(\pi q,b)$,
the offset $b$ obeys $b>\si$,
because $|k|\le N/2$ and $\si < \pi q/2$.
Thus each sum satisfies the conditions of
Lemma~\ref{l:sincwarp}, stated and proved in the next subsection,
which bounds it by
\be
|S_\si(\pi q ,b)| \;\le\; \frac{5}{2\pi q\sqrt{\al}} + 5
\;<\;
\frac{1}{q\sqrt{\al}} + 5~,
\label{Sq}
\ee
uniformly over $k\in P$.
From the above formula for $(F\mbf{f})_k$,
applying the triangle inequality,
the fact that $|\sinc y|\le 1$ for $y\in\R$,
then the bound \eqref{Sq}, gives, for all $k\in P$,
$$
|(F\mbf{f})_k| \;\le\; 2q|S_\si(a ,b)| + 2q \;\le\; \frac{2}{\sqrt{\al}} + 12 q ~,
$$
so that
$$
\sum_{k\in P} (F\mbf{f})_k^2 \;\le\; \left(\frac{2}{\sqrt{\al}} + 12 q\right)^2 p
~.
$$
Defining $\mbf{v}\in\C^q$ as the central $q$ entries of $\mbf{f}$
just as in \eqref{vf},
and noting that the compact support of $f(t)$ makes
all other entries $f_j=0$,
we have $A\mbf{v} = (F\mbf{f})_{k\in P}$ and
$$
\|A\mbf{v}\| \;=\; \left(\sum_{k\in P}(F\mbf{f})_k^2 \right)^{1/2}
\;\le\;
\biggl(\frac{2}{\sqrt{\al}} + 12 q\biggr)\sqrt{p}
\;=\;
2\sqrt{N} (1 + 6 \sqrt{\al} q)~.
$$
We combine this with the simple bound $\|\mbf{v}\| \ge f_0 = f(0) = I_0(\si)-1$,
and recall $\si_\tbox{min}(A) \le \|A \mbf{v}\|/\|\mbf{v}\|$
for any $\mbf{v}\neq\mbf{0}$, to complete the proof.
\end{proof}

\bfi[t] % fffffffffffffffffffffffffffffffffffffffffffffffffffffffffffffff
\ig{width=2.4in}{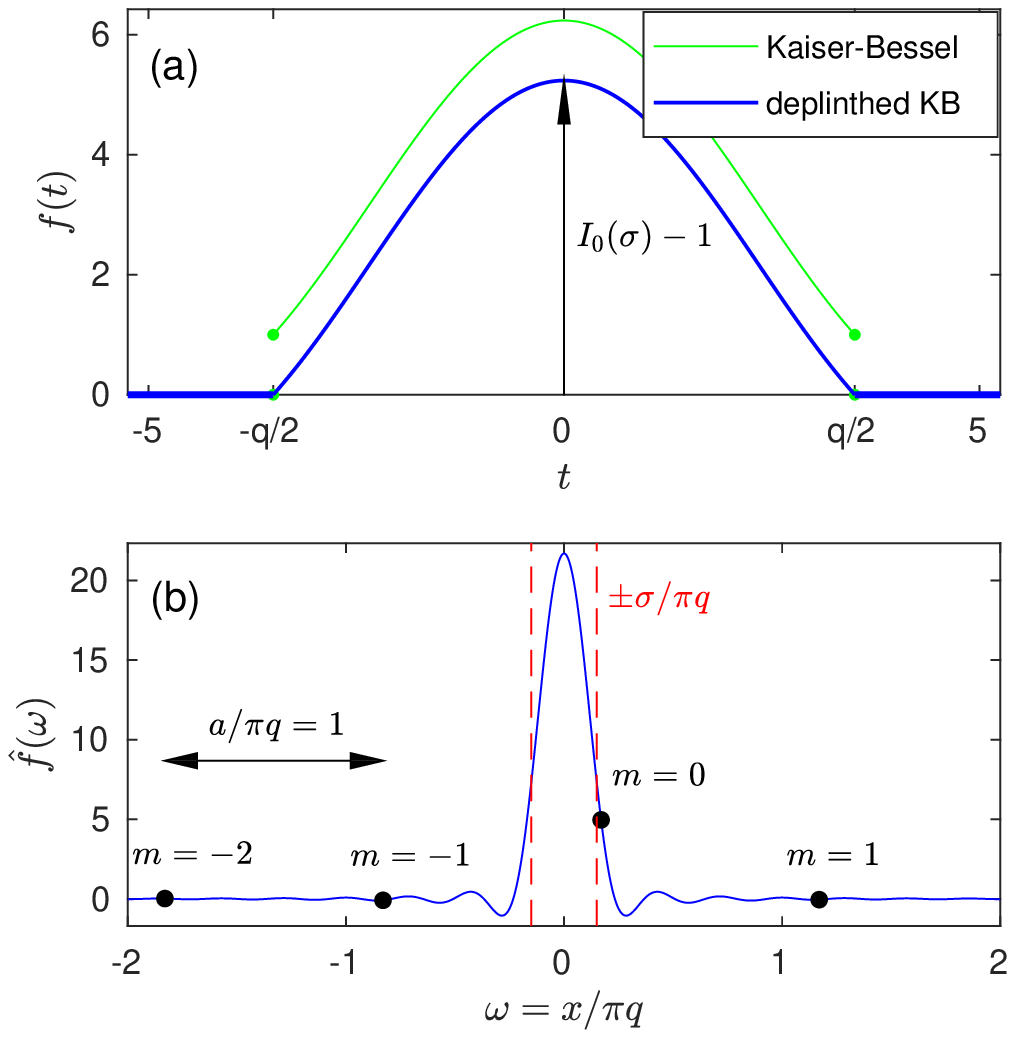}
\hfill
\ig{width=2.5in}{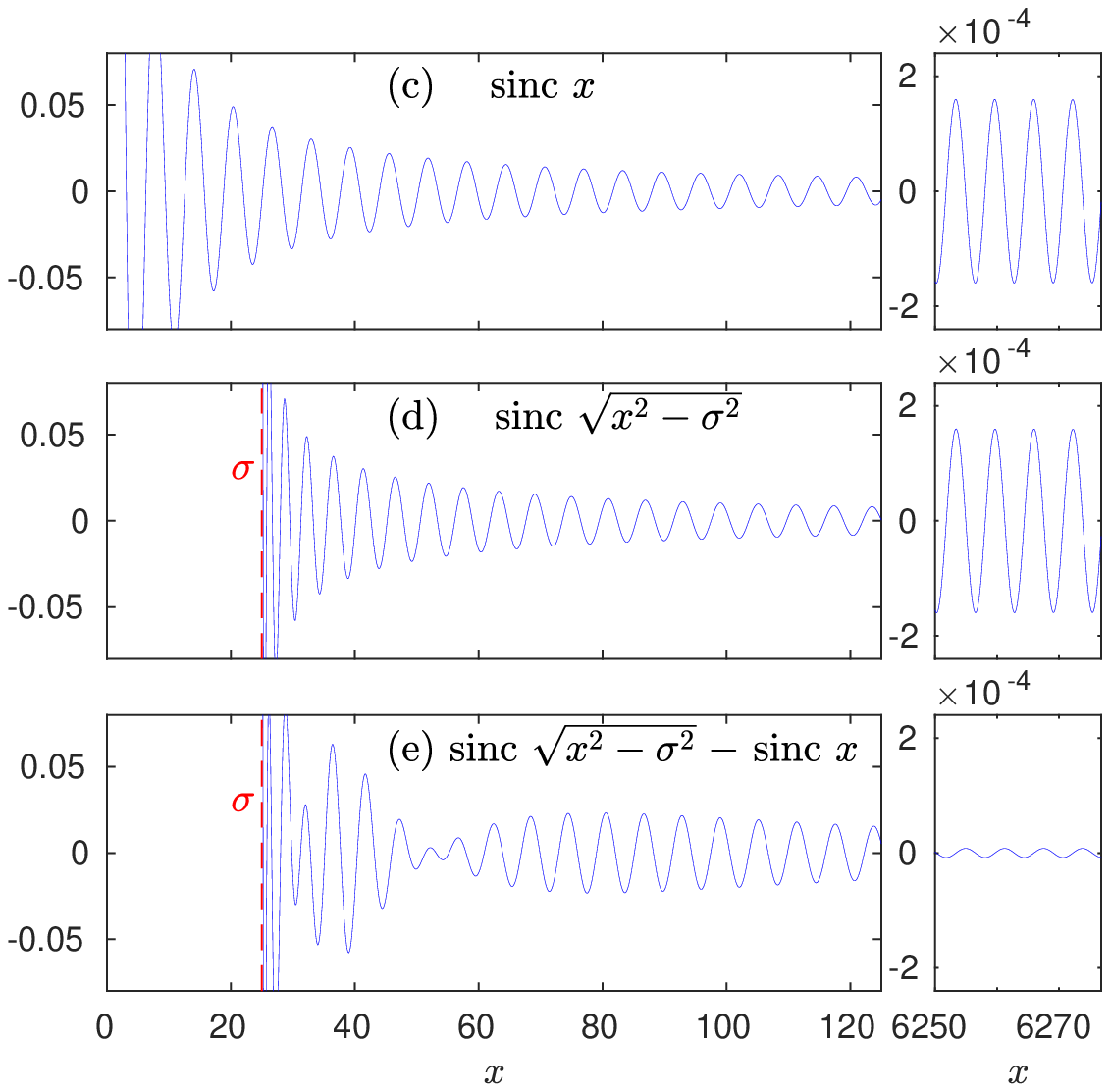}
\ca{
  Functions key to proving Theorem~\ref{t:kb}:
  deplinthed Kaiser--Bessel Fourier transform pair (left panels),
  and various sinc functions (right and far-right panels).
  (a) shows $f(t)$ from \eqref{DKBpair} (blue), and the KB function
  ($\phi(t)$ in \eqref{KBpair} except scaled to have support $[-q/2,q/2]$,
  green).
  Here $q=7$ (chosen small for easy visualization),
  and $\al=0.7$.
  (b) shows $\hat f(\omega)$ from \eqref{DKBpair} (blue),
  the cutoff (dotted red),
  and the arithmetic progression $\omega = k/N + m$ (black dots)
  as in \eqref{DKBFf}, for $k$ just above $(1-\al)N/2$.
  (c,d,e) compares the sinc, ``warped'' sinc, and their difference.
  ($x$ will take the values $x=am+b$). Here $\sigma=25$.
  %larger than in (a--b).
  The narrow plots on the far right are for higher $x$,
  and a zoomed vertical scale, showing that the difference (e) has the fastest
  asymptotic decay ($x^{-2}$ as opposed to $x^{-1}$).
}{f:dkb}
\efi

% SSSSSSSSSSSSSSSSSSSSSSSSSSSSSSSSSSSSSSSSSSSSSSSSSSSSSSSSSSSSSSSSSSSSSSSSSSS
\subsection{Technical lemmas on one-sided warped sinc sums}

The one-sided sum $S_\si(a,b)$ defined by \eqref{S}
involves the sinc of each frequency in an algebraic progression
$\{am+b\}_{m\ge 0}$, but also the sinc of each frequency in a ``warped''
sequence $\{\sqrt{(am+b)^2-\si^2}\}_{m\ge 0}$.
The proof of Theorem~\ref{t:kbsig} relied on our ability to bound algebraic
sums over the difference between warped and unwarped sinc functions,
expressed by the main Lemma~\ref{l:sincwarp} below.
This lemma is nontrivial, because the sum of
a single sinc over a general algebraic progression is at best conditionally
convergent (because its tail decays with power $-1$),
and may not even be finite (see Remark~\ref{r:offset}).
One might hope that the {\em difference} between
the sinc functions in \eqref{S} 
eventually has more rapid decay, because the difference in
their two arguments tends to zero as $m\to\infty$.
This is true---illustrated by Figure~\ref{f:dkb}(c--e)---and
is the idea behind its proof,
which occupies the rest of this section.

\begin{rmk} \label{r:fourmont}  % rrrrrrrrrrrrrrrrrrrrrrrrrrrrrrrr
  Our technique adapts several steps from the thesis of Fourmont
  \cite[Sec.~2.5]{fourmontthesis},
  who showed how phased sums over a sinc whose argument is a warped
  algebraic progression (with offset $b=na$, $n\in\N$)
  can be bounded by the phased sum over the unwarped sinc.
  This requires decomposing the sinc function into numerator
  and denominator factors, and showing how the effect of warping each factor
  becomes small for large enough $m$.
  His task was to bound the aliasing error in the
  non-uniform FFT algorithm when using \eqref{KBpair} as a spreading
  kernel \cite{fourmont}.
  % the latter can be bounded by a known fourier series.
  Our task is in one way simpler because there is no phase,
  but we also need to handle general $b$.
\end{rmk}

\begin{lem} \label{l:sincwarp}  % lllllllllllllllllllllllllllllllllllllll
  Let $a\ge \pi$, $\si\in(0,a/2)$, and $b>\si$.
  Then
  \be
  \big| S_\si(a,b) \big | =
  \biggl|
  \sum_{m\ge 0} \sinc \sqrt{(am+b)^2-\si^2} - \sinc (am+b)
  \biggr|
  \; \le \;
  \frac{5}{2 a\sqrt{\al}} + 5 ~,
  \label{sincwarp}
  \ee
  where $\al = 1 - 2\si/a \in (0,1)$.
\end{lem}

Note that we use the symbol $\al$ here 
since it is consistent with \eqref{sig}
when $a=\pi q$, as occurs when this lemma is applied above in \eqref{Sq}.

\begin{proof}
  From now on we use $x= a m +b$ to denote a frequency in
  the progression.
  Then, since $x>0$,
  \bea
  S_\si(a,b)
  &=& \sum_{m\ge0} \frac{\sin \sqrt{x^2-\si^2}}{\sqrt{x^2-\si^2}}
  - \frac{\sin x}{x}
  \nonumber \\
  &=& \sum_{m\ge0} \sin \sqrt{x^2-\si^2}
    \biggl(\frac{1}{\sqrt{x^2-\si^2}}-\frac{1}{x}\biggr)
    + \frac{\sin \sqrt{x^2-\si^2} - \sin x}{x}
    ~,
    \eea
    which is Fourmont's decomposition.
  By the triangle inequality, and using that $|\sin y|\le 1$,
  $$
  \big|S_\si(a,b)\big|
  \;\le \; \sum_{m\ge0}
  \biggl| \frac{1}{\sqrt{x^2-\si^2}}-\frac{1}{x} \biggr|
  + \sum_{m\ge0} \biggl| \frac{\sin \sqrt{x^2-\si^2} - \sin x}{x} \biggr|
  ~.
  $$
  We apply the ``denominator warp'' Proposition~\ref{p:denomwarp} stated below (using $y = a/2$ and the
  definition of $\al$)
  to the first summand.
  We also apply the ``numerator warp'' Lemma~\ref{l:sinwarp} stated below,
  which applies since $a\ge \pi$,
  to bound the second sum by $5$.
  We then split off the first term of the remaining sum and bound the rest by
  an integral:
  \bea
  \big|S_\si(a,b)\big|
  &\le&
  \frac{\si^2}{\sqrt{\al}} \sum_{m\ge0} \frac{1}{x^3} + 5
  \; \le \;
  \frac{\si^2}{\sqrt{\al}b^3}
  +
  \frac{\si^2}{\sqrt{\al}}
  \sum_{m>0} \frac{1}{x^3} + 5
  \nonumber \\
  &\le & 
  \frac{2}{a\sqrt{\al}} + \frac{\si^2}{\sqrt{\al}}\cdot\frac{1}{a} \int_b^\infty \frac{dx}{x^3} + 5
  \; \le \;
   \frac{2}{a\sqrt{\al}} +  \frac{1}{2a\sqrt{\al}} + 5~,
  \eea
  where in the final two steps we used $1/b < 2/a$ and $\si/b < 1$.
\end{proof}

The rest of this subsection is devoted to the required bounds on the
effect of warping on the numerator and denominator of the sinc function.
The first estimate shows that the effect of warping the denominator
is $\bigO(x^{-3})$ as frequency $x\to\infty$, with an explicit constant.

\begin{pro}[Denominator warp]
  \label{p:denomwarp}  % ppppppppppppppppppppppppppppppppppppppppppppppppppp
  % here y will get set to a/2.
  Fix $\si>0$ and $y>\si$, then
  \be
  \frac{1}{\sqrt{x^2-\si^2}} - \frac{1}{x} \;\le\;
  \frac{\si^2}{\sqrt{1-\si/y}} \cdot
  \frac{1}{x^3}
  ~,\qquad\mbox{ for all } x\ge y~.
  \ee
\end{pro}
\begin{proof}
  Since $x>\si$, expanding the left side gives
  $$
  \frac{1 - \sqrt{1 - (\si/x)^2}}{x\sqrt{1 - (\si/x)^2}}
  = \frac{(\si/x)^2}{x\sqrt{1 - (\si/x)^2}(1 + \sqrt{1 - (\si/x)^2})}
  \le \frac{\si^2}{\sqrt{1-\si/x}}
    \cdot
  \frac{1}{x^3}~,
  $$
  and applying $x\ge y$ to the first factor gives the result.
%  where, since $\si/x \le 1-\al$, one may choose the constant
%  $$
%  c(\si,a) = \frac{\si^2}{\sqrt{1 - (1-\al)^2}(1 + \sqrt{1 - (1-\al)^2})}
%  \le \frac{\si^2}{\sqrt{\al}\sqrt{2-\al}}
%  \le \frac{a^2}{4\sqrt{\al}}~.
%$$
\end{proof}

The following shows that the effect of warping the numerator on the sum is
uniformly bounded over the allowed set of parameters $\si, a, b$.
Its proof needs several results which complete the subsection.

\begin{lem}[Numerator warp]
  \label{l:sinwarp}   % llllllllllllllllllllllllllllllllllllllllllllllllllll
  Let $a\ge \pi$, $\si \in (0,a/2)$, and $b\ge a/2$.
  Then, using the abbreviation $x=am+b$,
  \be
  \sum_{m\ge0} \biggl| \frac{\sin \sqrt{x^2-\si^2} - \sin x}{x} \biggr|
  \;\le\;
  5~.
  \label{sinwarp}
  \ee
\end{lem}
\begin{proof}
  Fixing $\si$, will make frequent use of the frequency warping deviation function,
  \be
  R(x):=x -\sqrt{x^2-\si^2}~,\qquad x\ge\si~.
  \label{R}
  \ee
  Now, noting $\sin \sqrt{x^2-\si^2} =
  \sin (x - R(x))$, applying the addition formula,
  subtracting $\sin x$ from both sides
  and dividing by $x$ gives
  $$
  \frac{\sin \sqrt{x^2-\si^2} - \sin x}{x} =
  \sin x\frac{\cos R(x) - 1}{x} - \cos x \frac{\sin R(x)}{x}
  ~.
  $$
  Applying the triangle inequality termwise, and bounds on sin and cos, gives
  $$
  \sum_{m\ge 0}
  \biggl| \frac{\sin \sqrt{x^2-\si^2} - \sin x}{x} \biggr|
  \; \le \;
  \sum_{m\ge 0}
  \biggl|\frac{\cos R(x) - 1}{x}\biggr| +
  \sum_{m\ge 0}
  \biggl|\frac{\sin R(x)}{x}\biggr|
  \;\le\; 3 + 2
  ~,
  $$
  by Lemmas~\ref{l:cosR1} and \ref{l:sinR} respectively.
\end{proof}

\begin{lem}
  \label{l:sinR}   % llllllllllllllllllllllllllllllllllllllllllllllllllll
  Let $a\ge \pi$, $\si \in (0,a/2)$, and $b\ge a/2$.
  Then, with $R(x)$ as in \eqref{R}, and using the abbreviation $x=am+b$,
  \be
  \sum_{m\ge0} \biggl| \frac{\sin R(x)}{x} \biggr|
  \;\le\;
  2~.
  \label{sinR}
  \ee
\end{lem}
\begin{proof}
  By Proposition~\ref{p:Rbnd} below, $R(x) = \bigO(x^{-1})$
  so that the tail of the sum is $\bigO(x^{-2})$ and hence summable.
  However, this only becomes useful for sufficiently large $m$.
  Thus the idea will be to split the sum at index
  $m_0 := \lceil \si^2/a\rceil \ge 1$,
  chosen so that, given the hypotheses on $\si$, $a$, and $b$,
  \be
  R(x) \le \frac{\si^2}{x} \le 1
  ~, \qquad \mbox{ for all } m\ge m_0~.
  \label{Rintail}
  \ee
  This follows easily from the fact that $x = am+b \ge am_0+b \ge \si^2+b$,
  that $b>\si$, and from Proposition~\ref{p:Rbnd}.
  For the first $m_0$ terms we
  use the crude bound $|\sin R(x)|\le 1$,
  but use $\sin y \le y$ for $0<y\le 1$ and \eqref{Rintail}
  in the tail sum, and get
  \be
  \sum_{m\ge0} \biggl| \frac{\sin R(x)}{x} \biggr|
  \;\le\;
  \sum_{m=0}^{m_0-1} \biggl| \frac{\sin R(x)}{x} \biggr| +
  \sum_{m\ge m_0} \biggl| \frac{\sin R(x)}{x} \biggr|
  \;\le\;
  \sum_{m=0}^{m_0-1} \frac{1}{x} + 
  \si^2 \sum_{m\ge m_0} \frac{1}{x^2}~.
  \label{sinRsplit}
  \ee
  After splitting off its first term, we
  bound the rest of the finite sum by an integral with upper limit
  $a(m_0-1)+b<\si^2+b$ 
  \bea
  \sum_{m=0}^{m_0-1} \frac{1}{am+b}
  &\le&
  \frac{1}{b} + \frac{1}{a}\int_b^{b+\si^2} \frac{dx}{x}
  \;\le\;
  \frac{2}{a} + \frac{1}{a}\log(1+\si^2/b)
  \;\le\;
  \frac{2}{a} + \frac{1}{a}\log(1+a/2)
  \nonumber \\
  &\le &  \frac{2}{\pi} + \frac{1}{\pi}\log(1+\pi/2)
  \;\le\; 1 ~,               % actually  < 0.938
  \label{harm0}
  \eea
  where the replacement of $a$ by its lower limit $\pi$
  is justified by checking that the function of $a$ has negative
  derivative for all $a>0$.
  The infinite sum in \eqref{sinRsplit} is similarly bounded
  by an integral
  $$
  \si^2\sum_{m\ge m_0} \frac{1}{(am+b)^2}
  \;\le\;
  \frac{\si^2}{(am_0+b)^2} + \frac{\si^2}{a}\int_{b+\si^2}^\infty \frac{dx}{x^2}
  \;\le\;
  \frac{\si^2}{a^2} + \frac{1}{a}     % 1/4 + 1/pi
  \;\le\; 1~,
  %\label{tail0}
  $$
  using $\si \le a/2$ and $a\ge \pi$.
  Adding the two above bounds completes the proof.
\end{proof}

\begin{lem}
  \label{l:cosR1}   % llllllllllllllllllllllllllllllllllllllllllllllllllll
  Let $a\ge \pi$, $\si \in (0,a/2)$, and $b\ge a/2$.
  Then, with $R(x)$ as in \eqref{R}, and using the abbreviation $x=am+b$,
  \be
  \sum_{m\ge0} \biggl| \frac{\cos R(x) - 1}{x} \biggr|
  \;\le\;
  3~.
  \label{cosR1}
  \ee
\end{lem}
\begin{proof}
  The proof is very similar to that of Lemma~\ref{l:sinR}.
  We choose $m_0$ in the same way.
  Then for all $m \ge m_0$, since $R(x)\le1$ then, by Taylor's theorem
  followed by Proposition~\ref{p:Rbnd},
  $$
  | \cos R(x) - 1| \le \frac{R(x)^2}{2} \le \frac{\si^4}{2x^2}~.
  $$
  Again splitting the sum, and bounding the first $m_0$ terms
  via $|\cos R(x) - 1|\le 2$, and the rest via the above formula,
  $$
  \sum_{m\ge0} \biggl| \frac{\cos R(x) - 1}{x} \biggr|
  \;\le \;
  \sum_{m=0}^{m_0-1} \frac{2}{x} +
  \frac{\si^4}{2} \sum_{m\ge m_0} \frac{1}{x^3}~.
  $$
  The first sum is no more than 2, using \eqref{harm0}.
  The second sum we bound by its first term plus an integral to get
  $$
  \frac{\si^4}{2}\sum_{m\ge m_0} \frac{1}{(am+b)^3}
  \;\le\;
  \frac{\si^4}{2(am_0+b)^3} + \frac{\si^4}{2a}\int_{b+\si^2}^\infty \frac{dx}{x^3}
  \;\le\;
  \frac{\si^4}{(am_0+b)^2} \cdot \frac{1}{2a}
  + \frac{\si^4}{2a}\cdot \frac{1}{2\si^4}        % 3/4pi = 0.24..
  \;\le\; 1~,
  $$
  using $\si^2 < am_0+b$, $\si \le a/2$ and $a\ge \pi$.
  Adding these two bounds finishes the proof.   
\end{proof}

Finally, the above two proofs relied on the following fact that
the effect of warping the numerator is $\bigO(x^{-1})$.
\begin{pro}\label{p:Rbnd} % pppppppppppppppppppppppppppppppppppp
  Let $\si\ge 0$ and $R(x) := x-\sqrt{x^2-\si^2}$ as in \eqref{R}.
  Then $R(x) \le \si^2/x$ for all $x \ge \si$.
\end{pro}
\begin{proof}
  Since $\si<x$, we expand
  $$
  R(x) = x\big(1-\sqrt{1-(\si/x)^2}\big) =
  \frac{x\cdot(\si/x)^2}{1+\sqrt{1-(\si/x)^2}}
  \le \frac{\si^2}{x}~.
  $$
\end{proof}

% CCCCCCCCCCCCCCCCCCCCCCCCCCCCCCCCCCCCCCCCCCCCCCCCCCCCCCCCCCCCCCCCCCCCCCCCCCC
\section{Proof of improved rate for small $\alpha$ and $\beta$
  (Theorem~\ref{t:corner})}
\label{s:corner}

As with the other two main theorems,
Theorem~\ref{t:corner} combines
a lower bound on $\sigma_1(A)$ (Proposition~\ref{p:s1})
and an upper bound on $\sigma_\tbox{min}(A)$, in this case the following lemma.

\begin{lem}  % tttttttttttttttttttttttttttttttttttttttttttttttttttttttttttttt
  Let $A$ be a cyclically contiguous $p\times q$ submatrix of the
  $N\times N$ discrete Fourier matrix $F$, with $1 < q\le p < 4N/e\pi + 1$.
  Then
  \be
  \sigma_\tbox{min}(A)
  \;\le\;
  \frac{2\sqrt{pq}}{1- (e\pi(p-1)/4N)}
  \left(\frac{e \pi(p-1)}{4N}\right)^{q-1}~.
  \label{corner}
  \ee
  \label{l:corner}
\end{lem}

\begin{proof}
  Following \cite[Lem.~3.2]{oneil07} we start with the Bessel--Chebyshev
  expansion
  \be
  e^{ixt} = J_0(x) + \sum_{n=1}^\infty 2 i^n J_n(x) T_n(t)~,
  \qquad \mbox{ uniformly for } \; x\in\R~, |t|\le 1~,
  \label{cheb}
  \ee
  arising by inserting $t=\cos \theta$ into the Jacobi--Anger
  expansion $e^{ix \cos\theta} = \sum_{n\in\Z}i^n J_n(x) e^{in\theta}$.
  The submatrix elements can be generated by
  sampling this kernel on the product of the regular grids 
  $x_j = (-1 + 2j/(p-1))W$, $j=1,\dots,p$ and 
  $t_k := -1 + 2k/(q-1)$, $k=1,\dots,q$.
  Scaling the $x$-domain size by setting $W = \pi(q-1)(p-1)/2N$ insures that
  the matrix with elements $A_{jk} = e^{i x_j t_k}$ is, up to irrelevant
  left and right multiplication by diagonal unitary matrices,
  equal to any given $p\times q$ contiguous submatrix of the
  size-$N$ Fourier matrix $F$.
  Defining for $n=0,1,\dots$
  the sequences of vectors $\mbf{u}_n\in\C^p$ and $\mbf{v}_n\in\C^q$,
  with elements $(\mbf{u}_0)_j \equiv 1$, $(\mbf{u}_n)_j:=2i^nJ_n(x_j)$
  for $n=1,2,\dots$, and $(\mbf{v}_n)_k:=T_n(t_k)$ for $n=0,1,\dots$,
  we rewrite the matrix of samples of \eqref{cheb} as the sum of
  rank-1 outer products
  $$
  A = \sum_{n=0}^\infty \mbf{u}_n \mbf{v}_n^\ast~.
  $$
  Since $q \le p$, then
  $\sigma_\tbox{min}(A) = \sigma_q(A)$.
  The rank approximation theorem (e.g., \cite[Thm.~2.5.3]{golubvanloan})
  bounds
  \bea
  \sigma_\tbox{min}(A) &\;\le\;&
  \biggl\|A - \sum_{n=0}^{q-2} \mbf{u}_n \mbf{v}_n^\ast\biggr\| \;=\;
  \biggl\|\sum_{n=q-1}^\infty \mbf{u}_n \mbf{v}_n^\ast\biggr\|
  \;\le\; \sum_{n=q-1}^\infty \|\mbf{u}_n\| \|\mbf{v}_n\|
  \nonumber \\
  &\;\le\;&
  2\sqrt{pq} \sum_{n=q-1}^\infty \max_{0\le x \le W} |J_n(x)|
\;\le\;
  2\sqrt{pq} \sum_{n=q-1}^\infty [g(W/n)]^n
  ~,
  \qquad \mbox{ if } q-1 > W~,
  \nonumber
  \eea
  where in the penultimate step we used $|T_n(t)|\le 1$,
  and in the last step Siegel's bound for $J_n(nz)$ \cite[10.14.5]{dlmf} where
  $g(z) := z e^{\sqrt{1-z^2}}/(1+\sqrt{1-z^2})$, for $z = x/n \le 1$.
  Note that $q-1>W$ is equivalent to $p<2N/\pi +1$ which holds by
  the hypotheses of the lemma.
  We now observe that $g(z)\le ez/2$ in $0<z\le1$, bound $g(W/n)$ by
  $g(W/(q-1))$, and recall $W/(q-1) = \pi(p-1)/2N < 1$, to get
  $$
  \sigma_\tbox{min}(A)
  \;\le\;
  2\sqrt{pq} \sum_{n=q-1}^\infty \left(\frac{e \pi(p-1)}{4N}\right)^n~,
  $$
  a bounded geometric sum when $(p-1)/N < 4/e\pi$,
  a hypothesis of the lemma,
  giving \eqref{corner}.
\end{proof}

\begin{rmk}\label{r:cornerrate}
  We initially tried the low-rank expansion of $e^{ixt}$ resulting from the
  Taylor series for $e^z$,
  following \cite[Lem.~1]{candes07},
  but found that this gave a rate replacing the exponential term in
  \eqref{corner} with
  $[e\pi(p-1)/2N]^{q-1}$, which is exactly half the rate in \eqref{corner},
  and even weaker than \eqref{liliao}.
  It is therefore possible that yet other low-rank expansions,
  such as the double-Chebyshev \cite[App.~A]{townsendnufft},
  could further improve the rate in the corner region $\alpha,\beta\ll 1$.
\end{rmk}

\bfi[t!] % fffffffffffffffffffffffffffffffffffffffffffffffffffffffffffffffffff
\centering
\ig{width=5.2in}{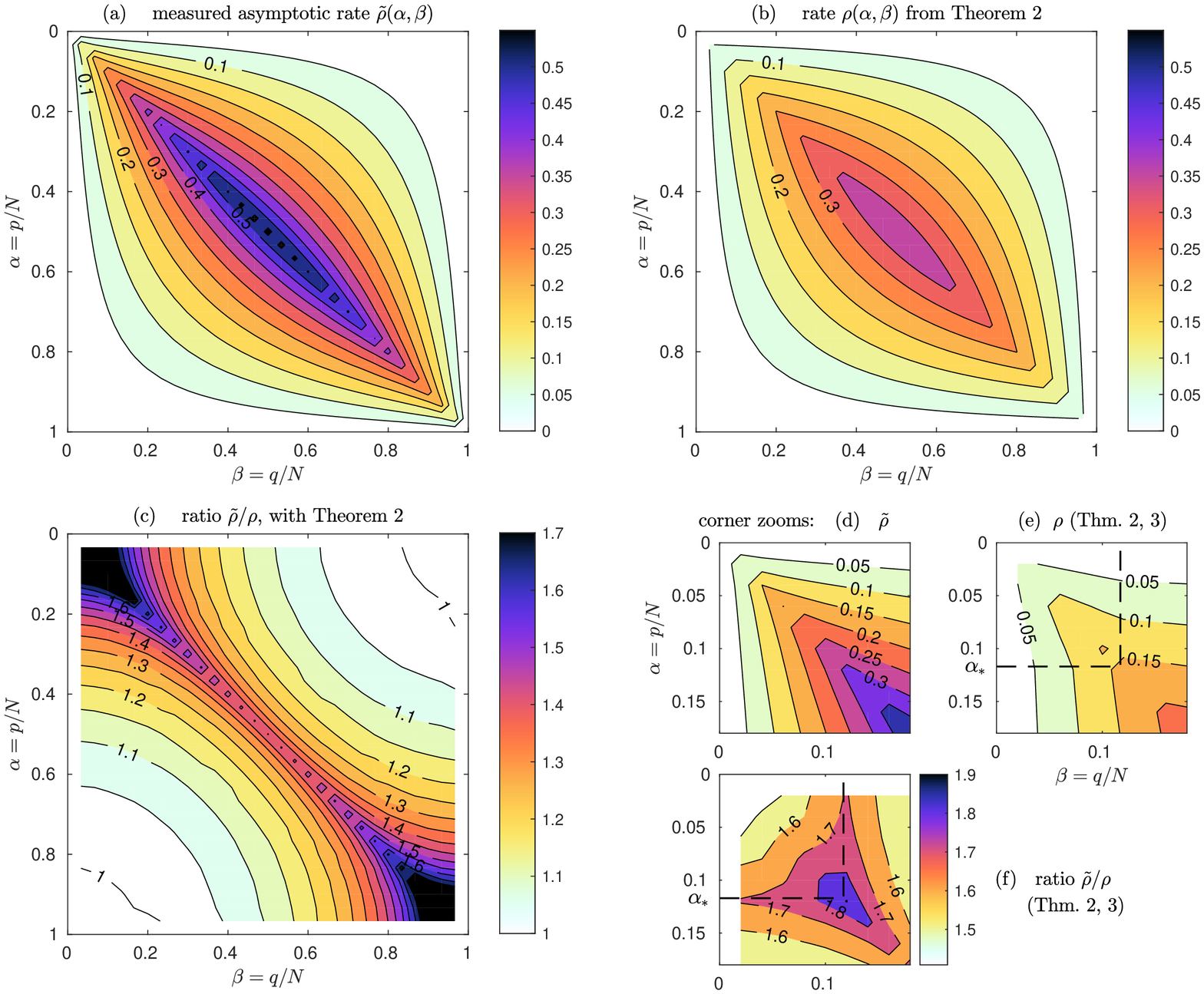}
\ca{
  Comparison of exponential growth rates with $N$
  of the condition number of a
  $p\times q$ submatrix of the size-$N$ Fourier matrix, as a function of
  shape parameters $\al:=p/N$ and $\bt:=q/N$.
  The $\al$ axis increases downward to match the usual matrix
  convention.
  (a) shows contours of the numerically-measured rate, fitting the
  model $\cond(A) \approx C e^{\tilde\rho(\al,\bt) N}$, for $(\al,\bt)$ sampled
  on a grid of spacing $1/30$; see Remark~\ref{r:num}.
  (b) shows the lower bound $\rho=\frac{\pi}{2}[\min(\alpha,\beta) - \alpha\beta]$
  from Theorem~\ref{t:kb}.
  (c) shows the ratio of (a) to (b), i.e., the sharpness of the theorem.
  The breakup into small dots along the diagonals of (a) and (c) are
  contouring artifacts.
  Panels (d,e,f) show a zoom around the origin of panels (a,b,c) respectively,
  on a grid of spacing $1/50$,
  but also applying Theorem~\ref{t:corner} in the corner $(0,\al_\ast)^2$
  (shown dotted),
  where it is stronger than Theorem~\ref{t:kb}.
}{f:rates}
\efi

% NNNNNNNNNNNNNNNNNNNNNNNNNNNNNNNNNNNNNNNNNNNNNNNNNNNNNNNNNNNNNNNNNNNNNNNNNNN
\section{Discussion}
\label{s:disc}

\subsection{Numerical study of sharpness of exponential rate bounds}
\label{s:num}

Figure~\ref{f:nearsymm} showed $\cond(A)$ for various fixed submatrices $A$,
at small $N$.
However, Theorems~\ref{t:m}--\ref{t:corner}
may be seen as lower bounds on the
{\em asymptotic} exponential growth rate with respect to $N$,
for fixed shape $(\al,\bt)$, as summarized in Table~\ref{t:sum}.
Thus we numerically measured the {\em empirical} asymptotic growth rate
$\tilde\rho(\al,\bt)$
along the family of submatrices sharing a given $(\al,\bt)$, but growing $N$, to
see how close our lower bounds were to this rate.
The result is Figure~\ref{f:rates}.
Comparing its panels (a) and (b) shows that (the symmetrized)
Theorem~\ref{t:kb} captures well
the main features of the empirical rate.
Indeed, the ratio $\tilde\rho/\rho$, plotted in (c),
is less than 1.2 in 54\% of the area of the square $(0,1)^2$.
This ratio approaches 1
in the neighborhoods of $(1,0)$ and $(0,1)$,
as expected since Theorem~\ref{t:kb} has a sharp rate there (Remark~\ref{r:prolrate}).
However, in quite a tight band around the diagonal $\al=\bt$ the
empirical rate exceeds this lower bound by a significant factor.
This factor is about 1.45
at $(1/2,1/2)$, and along the diagonal grows slowly but without bound as
the corners $(0,0)$ and $(1,1)$ are approached.

However, we know that the rate of Theorem~\ref{t:corner} beats
that of Theorem~\ref{t:kb}
inside the corner $(0,\alpha_\ast)^2$, where $\alpha_\ast\approx 0.117$
(see Section~\ref{s:res}).
Hence, in Figure~\ref{f:rates}(d--f) we zoom in to the neighborhood of this
region, comparing the empirical rate to the
stronger of these two theorems
(their cross-over at $\max(\al,\bt)=\alpha_\ast$ is visible as
kinks in the contour lines in (e) and (f)).
As panel (f) shows, the ratio now appears {\em uniformly} bounded,
reaching its maximum at $\al=\bt=\alpha_\ast$.
An estimate of this maximum is given by the ratio at $(0.12,0.12)$, which
is less than $1.91$.
As $\bt\to0$, the ratio peaks at a value about $1.7$ at $\al=\al_\ast$.
As $(\al,\bt)\to(0,0)$, the ratio drops.
Remark~\ref{r:prolrate} on the prolate
asymptotic suggests that
the ratio tends to 1 (logarithmically slowly) if approached along the axes.
%apparently tending to a limit %of about $1.3$ at the origin.
%This indicates that the low-rank method of Theorem~\ref{t:corner},
%while not sharp, is not too far off in this corner region.

\begin{rmk}[measuring the empirical rate $\tilde\rho(\al,\bt)$]\label{r:num}
  Generating Figure~\ref{f:rates} needed an accurate
  measurement of the constant $\tilde\rho$ in the
  asymptotic model $\cond(A) \sim ce^{\tilde\rho N}$ as $N\to\infty$.
  Recall that here $A$ takes a sequence of {\em integer} sizes $p\times q$
  with $p/N=\al$ and $q/N=\bt$ fixed (and, thus, rational).
  This places constraints on the allowable $p$ and $q$, making the
  measurement subtle.
  As a example, to estimate $\tilde\rho(0.5,0.49)$ the smallest
  possible case of $A$ is a $50\times 49$ submatrix with $N=100$,
  and yet already $\cond(A)>10^{16}$, exceeding
  the value reliably measurable in double-precision arithmetic.
  
  There are two ways out of this quandary: either compute SVDs via a
  higher precision arithmetic library, or limit oneself to $\al$ and
  $\bt$ with small rational denominators. We opted for the
  latter.  This led to the choices of $1/30$ and $1/50$ for the grid
  spacings in Figure~\ref{f:rates}, which were adequate.  We wrote a
  simple code which, given such a pair $(\al,\bt)$, uses a bisection
  search along the family of $A$ with this shape, identifying the
  largest $A$ for which $\cond(A)$ does not exceed an upper limit of
  around $10^{16}$. This data point is combined with another taken
  from the $A$ nearest to half this size (giving $\cond(A)\sim 10^8$,
  bypassing any pre-asymptotic growth).
  The $\tilde\rho$ returned is the slope between (the log of) these two
  points.
  The accuracy is around $\pm0.02$, estimated by comparing
  equivalent rational forms for $(\al,\bt)$.  The code ran in MATLAB
  R2017a on a laptop with Intel i7 CPU, taking a few seconds to
  generate all figures in this paper.  Certainly a more elaborate
  arbitrary-precision code could be built, but the above served our
  purposes well.
\end{rmk}

The code described above, plus those generating the other figures in this paper,
can be found at 
{\tt https://github.com/ahbarnett/fourier-submat}

% ssssssssssssssssssssssssssssssssssssssssssssssssssssssssssssssssssssssssss
\subsection{Symmetry and near-symmetry in the $(\alpha,\beta)$ plane}
\label{s:symm}

The empirical rate plot Figure~\ref{f:rates}(a)
appears to have two symmetries (i.e.\ the $D_2$ dihedral group).

The first is that $\cond(A)$ is invariant to the diagonal reflection
$(\alpha,\beta) \mapsto (\beta,\alpha)$,
which follows immediately from the fact that
a submatrix of swapped dimensions is the adjoint, and $\cond(A)^\ast = \cond(A)$.

The second symmetry in Figure~\ref{f:rates}(a)
is that the rate also appears
invariant under $(\alpha,\beta) \mapsto (1-\alpha,1-\beta)$,
i.e., inversion about $(1/2,1/2)$, or complementing both row and column
index sets.
However, understanding this is more subtle.
From the rate one might suspect that $\cond(A)$ itself is
also invariant under changing the
submatrix from size $p\times q$ to $(N-p)\times(N-q)$.
But this cannot be true, as the case $p=q=1$ shows:
the condition number of any $1\times 1$ matrix is unity,
and letting $A$ now be a $(N-1)\times(N-1)$ submatrix of $F$
gives $\sigma_1(A) = \sqrt{N}$ by the interlacing property
(e.g.\ \cite[Thm.~1]{thompson72}).
Choosing the omitted row and column to be $j=k=0$,
it is easy to check that $A$ acting on the vector of all ones is sent to
the negative of this vector, showing $\sigma_\tbox{min}(A) \le 1$
and $\cond(A) \ge \sqrt{N}$ (this turns out to be an equality), which
is certainly not unity!

The condition number of a submatrix of $F$
is {\em not} invariant with respect to complementing
its dimensions $p\times q$ to give $(N-p)\times(N-q)$,
yet as $N$ grows this becomes a {\em near} symmetry.
Figure~\ref{f:nearsymm} shows how strikingly rapid this is:
even at $N=8$ the near-symmetry is clear, and by $N=16$
the {\em failure} of symmetry is almost impossible to see.
This seems quite mysterious until we spot the following identity.

\begin{pro}  % ppppppppppppppppppppppppppppppppppppppppppppppppppppppppppppp
  Let $A$ be a $p\times q$ submatrix of $F$ with $p+q<N$,
  and let $D$ be the
  $(N-p)\times(N-q)$ submatrix of $F$ defined by the complement of the
  row and column index sets of $A$. Then
  \be
  \frac{\cond(A)}{\cond(D)} \; = \;
  \sqrt{1 - \frac{(\sigma_\tbox{\rm min}(C))^2}{N}}
  ~,
  \label{condrat}
  \ee
  where $C$ is the $(N-p)\times q$ submatrix of $F$ with the same
  column indices as $A$, and the complement of the row indices.
  \label{p:condrat}
\end{pro}
\begin{proof}
  If $p<q$ then we apply the diagonal reflection symmetry, so that
  we can from now take $p\ge q$.
  %The statement does not in fact need contiguity of index sets.
  We apply a permutation moving $A$ to the upper left position in $F$,
  which does not change $\cond(A)$ nor $\cond(D)$,
  then label the blocks as above, with the missing fourth block named $B$:
%acting on a trial vector for $A$
$$
F\vt{\mbf{v}}{\mbf{0}}
=
\left[\begin{array}{rr}A&B\\C&D\end{array}\right]\vt{\mbf{v}}{\mbf{0}}
=
\vt{\mbf{u}}{\mbf{w}}
~.
$$
The rest is simple singular value inequalities.
If $\mbf{v}$ is the normalized minimum right singular vector for $A$,
then because $A$ is square or ``tall'',
$\|\mbf{u}\|=\|A\mbf{v}\| = \sigma_\tbox{min}$ is as small as possible over
unit vectors $\mbf{v} \in \C^q$.
However, by isometry of $F$, $\|\mbf{u}\|^2 + \|\mbf{w}\|^2 = N$,
so that $\|\mbf{w}\| = \|C\mbf{u}\|$ is the largest possible, hence
$\mbf{v}$ is also the normalized maximum right singular vector for $C$.
Thus
\be
(\sigma_\tbox{min}(A))^2 + (\sigma_\tbox{max}(C))^2  = N~.
\label{AC}
\ee
Swapping the words ``minimum'' and ``maximum'' in this argument
requires additionally that $C$ not be ``fat'', which is true
by the assumption $q<N-p$.
This gives
\be
(\sigma_\tbox{max}(A))^2 + (\sigma_\tbox{min}(C))^2  = N~.
\label{CA}
\ee
We now reason similarly for the
blocks of $F^\ast = [A^\ast, C^\ast; B^\ast, D^\ast]$, multiplying it against
$[\mbf{0}; \mbf{v}]$ to sample its right two blocks.
By the assumption, $D^\ast$ is not fat, so may play the role of
$A$ in \eqref{AC}, giving
\be
(\sigma_\tbox{min}(D))^2 + (\sigma_\tbox{max}(C))^2  = N~.
\label{DC}
\ee
However, $C^\ast$ is fat, so the minimum of $\|C^\ast\mbf{v}\|$
over unit-norm $\mbf{v} \in \C^q$ is zero, leaving
\be
(\sigma_\tbox{max}(D))^2  = N~.
\label{0D}
\ee
Combining \eqref{AC} and \eqref{DC} shows
$\sigma_\tbox{min}(A)=\sigma_\tbox{min}(D)$.
Thus the left-hand side of  \eqref{condrat} is
$\sigma_\tbox{max}(A)/\sigma_\tbox{max}(D)$, and
inserting  \eqref{0D} and \eqref{CA} completes the proof.
\end{proof}

This proposition means that $\cond(A)$ for a submatrix shape
$(\alpha,\beta)$ in the triangle $\alpha+\beta < 1$
is {\em strictly} smaller than $\cond(A)$ for its
inverted shape $(1-\alpha,1-\beta)$.
However, the relative difference between the two condition numbers
vanishes as $C$, itself a submatrix of shape $(1-\alpha,\beta)$,
becomes highly ill-conditioned.
Since as $N$ grows this happens everywhere except $\alpha,\beta\approx 0$,
the mystery of the near-symmetry of Figure~\ref{f:nearsymm} is explained.
By the lower bound of
Theorem~\ref{t:m} or \ref{t:kb}
this occurs eventually for any $(\alpha,\beta)\in(0,1)^2$,
proving that the true {\em asymptotic} exponential growth rate
is exactly inversion symmetric.

Intriguingly, the rates in Theorems~\ref{t:m} and \ref{t:kb} also
obey this inversion symmetry; in their proofs this can be
traced to the fixed ``Heisenberg''
product of the sizes of the sets $\overline{P}$
and $Q$ (e.g.\ see Figure~\ref{f:idea}(b), and \eqref{sig}).
Yet Theorem~\ref{t:corner} cannot obey it, because its maximum $\al$ allowed
is less than $1/2$.
However, combining it with Proposition~\ref{p:condrat}
allows it to be also applied near the $(1,1)$ corner, as follows.
\begin{cor} % ccccccccccccccccccccccccccccccccccccccccccccccccccccccccccc
  The lower bound on exponential rates given in the last row of
  Table~\ref{t:sum} also apply with the replacement of
  $(\al,\bt)$ by $(1-\al,1-\bt)$, in the region $\al$, $\bt > 1-4/e\pi$.
\end{cor}   % ccccccccccccccccccccccccccccccccccccccccccccccccccccccccccc

% ccccccccccccccccccccccccccccccccccccccccccccccccccccccccccccccc
\subsection{Conclusions and open problems}

Fourier submatrices are quite elementary objects, arising in many
applications, yet until now there have been very few rigorous bounds on
their conditioning.
Our new non-asymptotic lower bounds on the condition
number of a $p\times q$ submatrix of the size-$N$ Fourier matrix
apply to all $N$ and, when combined with
the obvious $p\leftrightarrow q$ symmetry, essentially all $p$ and $q$.
All constants are explicit.
Interpreted as asymptotic results for $p$ and $q$ fixed fractions of $N$,
as $N\to\infty$, they give exponential lower bounds with
rates listed in Table~\ref{t:sum}.
Numerical study (Section~\ref{s:num}) shows that
their rates capture quite well the empirical growth,
which is also exponential, uniformly over shape space.
Section~\ref{s:symm} explained a non-obvious symmetry effect in this space.

As an example application of our results, in 1D Fourier extension
with the parameters $T=\gamma = 2$ recommended by \cite{adcock14},
the exponential growth rate%
\footnote{In the notation of \cite{adcock14}, $\kappa(\overline{A})$
  grows at least as fast as $e^{3\pi N / 2}$, up to algebraic prefactors.}
has the bound $\rho(1/2,1/4) \ge 3\pi/16$, by applying
Theorem~\ref{t:kb}.

Our methods are elementary: Theorem~\ref{t:m} used
a Gaussian trial vector, Theorem~\ref{t:kb} a more sophisticated
trial vector and quite detailed estimates,
while Theorem~\ref{t:corner}, applying only in the ``corner'' region
of small $\al$ and $\bt$, used the SVD rank approximation theorem.
In Theorem~\ref{t:kb}, we
achieved a rate that is {\em sharp} in the limit $(\al,\bt) \to (1,0)$,
or $(0,1)$, due to using the
(deplinthed) Kaiser--Bessel pair \eqref{DKBpair}. This shares
with the prolate spheroidal wavefunctions
an optimal frequency localization rate in $L^2(\R)$,
but is more accessible.
% may just be equiv to Rmk
We suspect that this is its first use as a pure analysis tool.

Our rate is {\em not} sharp long the axis $(\al,0)$, or $(0,\bt)$,
but possibly could be made so
by using a discrete prolate spheroidal sequence (DPSS) \cite{slepianV}
as the trial vector $\mbf{v}$.
This would require estimating {\em finite sums} over the
tails of $H(\omega)$ in \cite{slepianV}, for which only asymptotics are known.
It is also unclear whether this would simply duplicate \eqref{prolrate}.
Another approach would be to extend the methods of \cite{zhu17}.
Also see Remark~\ref{r:cornerrate}.
% explain that DPSS win over PSWF because they only have to be exp small
% in a chunk of the interval (even if periodized by the PSF) ?

This paper studied lower bounds in depth; we did not
address upper bounds at all, which seems to be quite an open area.
On this topic, for now we direct the reader to
\cite{bazan00},
and rather special cases in recent super-resolution literature:
\cite[Thm.~2, case $A=1$]{liliao18},
\cite[Thm.~3.2]{batenkov20},
and \cite[\S 4]{kunis19}.
These references also deal with the generalization
to non-uniform Fourier matrices (also see \cite{townsendnufft}).
In 2D and 3D applications, Kronecker products of Fourier submatrices
arise; we explore one such direction in \cite{epsteinhole}.

Our numerical study suggests fascinating open problems.
There appears to be a {\em universal} exponential rate as a function
of shape, as shown in Figure~\ref{f:rates}(a); what is it?
(Answering this would be equivalent to writing an
asymptotic form for the smallest P-DPSS eigenvalue \cite{grunbaum81,zhu17}.)
In particular, what phenomenon
explains the sudden spike in growth rate
around the diagonal, where the submatrix tends to become square?
% preasymp, power law as slice along q at p, N const?

\section{Acknowledgments}
The author benefited from helpful discussions with
Hannah Lawrence, Daan Huybrechs,
Charlie Epstein, David Barmherzig, and Alex Townsend, and from corrections by
Dominik Nagel.
The Flatiron Institute is a division of the Simons Foundation.

\appendix
\section{Proof of the Kaiser--Bessel Fourier transform pair}
Here we prove \eqref{KBpair}, in two steps.
The first will be to establish a related Fourier transform,
\be
\int_{-1}^1 J_0(b\sqrt{1-z^2})e^{ikz} \,dz \;=\; 2 \sinc \sqrt{k^2+b^2}
~, \qquad \mbox{ for } b\in\R, \; k\in\R~.
\label{Jpair}
\ee
This is equivalent to a formula in Gradshteyn--Ryshik
\cite[6.677(6)]{GS8}, although neither of the books cited therein
prove it or give a reference; the scent goes cold.
We prove it simply by {\em averaging a plane wave over the unit sphere}.
The second step will analytically continue this formula to
imaginary $b$.

\begin{proof}
  We insert the integral representation
\cite[(10.9.1)]{dlmf}
  of the Bessel function
  $J_0(s) = (2\pi)^{-1} \int_0^{2\pi} e^{is\sin \phi} d \phi$
  into the left-hand side of \eqref{Jpair}, then change variable
  $z=\cos\theta$, thus
  \bea
  &&\int_{-1}^1 J_0(b\sqrt{1-z^2})e^{ikz}\, dz \;=\;
  \frac{1}{2\pi}\int_{-1}^1 \int_0^{2\pi} e^{ib\sqrt{1-z^2}\sin \phi}e^{ikz} \,d\phi\,dz
  \nonumber \\
  && \qquad =\;
  \frac{1}{2\pi}\int_0^\pi \int_0^{2\pi} e^{ib\sin\theta \sin \phi + ik\cos\theta} \sin\theta \,d\phi\, d \theta
  \;=\; \frac{1}{2\pi}\int_{S^2} e^{i(0,b,k) \cdot (x,y,z)} \,dS_{(x,y,z)}
  \nonumber \\
  && \qquad = \;
    \frac{1}{2\pi}\int_{S^2} e^{i\sqrt{k^2+b^2}z} \,dS_{(x,y,z)}
    \;= \; \int_{-1}^1 e^{i\sqrt{k^2+b^2} z} \,dz
    \;= \; 2 \sinc \sqrt{k^2 + b^2}~,
  \nonumber
  \eea
  proving \eqref{Jpair}.
  The key step passing to the last line is invariance of the mean under
  rotation of the coordinate system, whose new $z$ axis points
  in the direction $(0,b,z)$.

  Although \eqref{Jpair} assumed real $b$, our second step
  shows that, fixing $k$,
  both sides are in fact entire with respect to $b$.
  This holds for the right-hand side
  because the square-root has singularities only at $b=\pm ik$,
  but these are removed by the entire function
  $\sinc$ having only even powers at its origin.
  The left-hand side is entire because the integrand is entire for each $z$
  and continuous in $(b,z)$, so one may apply \cite[Thm.~5.4, Ch.~2]{steinshakarchi}.
  By unique continuation the equality thus holds for all $b\in\C$,
  in particular $b=i\sigma$, so, using
  $I_0(s) = J_0(is)$ from \cite[(10.27.6)]{dlmf}, this gives
  $$
  \int_{-1}^1 I_0(\si\sqrt{1-z^2})e^{ikz} \,dz \;=\; 2 \sinc \sqrt{k^2-\si^2}~.
  $$
  Rewriting $k=2\pi\omega$ and $z=t$ gives \eqref{KBpair}.
\end{proof}

% BBBBBBBBBBBBBBBBBBBBBBBBBBBBBBBBBBBBBBBBBBBBBBBBBBBBBBBBBBBBBBBBBBBBBBBBBBBB
\bibliographystyle{abbrv}
\bibliography{alex}
\end{document}